\documentclass[11pt, a4paper]{amsart}

\usepackage{amscd,verbatim}
\usepackage{amssymb}

\usepackage{bm}

\usepackage{amssymb, amsmath}
\usepackage[all]{xy}
\usepackage[inline]{enumitem}
\usepackage{hyphenat}
\usepackage[colorlinks,linkcolor=blue,citecolor=blue,urlcolor=red]{hyperref}
\usepackage[dvipsnames]{xcolor}
\usepackage{mathtools}
\usepackage{tikz-cd}
\usepackage[labelformat=empty]{caption}
\usepackage{xfrac}
 \usepackage[ left=3cm, right=3cm, top = 2.8cm, bottom = 2.8cm ]{geometry}
\usepackage{enumitem}

\makeatletter
\def\@tocline#1#2#3#4#5#6#7{\relax
  \ifnum #1>\c@tocdepth 
  \else
    \par \addpenalty\@secpenalty\addvspace{#2}%
    \begingroup \hyphenpenalty\@M
    \@ifempty{#4}{%
      \@tempdima\csname r@tocindent\number#1\endcsname\relax
    }{%
      \@tempdima#4\relax
    }%
    \parindent\z@ \leftskip#3\relax \advance\leftskip\@tempdima\relax
    \rightskip\@pnumwidth plus4em \parfillskip-\@pnumwidth
    #5\leavevmode\hskip-\@tempdima
      \ifcase #1
      \or\or \hskip 2em \or \hskip 2em \else \hskip 3em \fi%
      #6\nobreak\relax
    \dotfill\hbox to\@pnumwidth{\@tocpagenum{#7}}\par
    \nobreak
    \endgroup
  \fi}
\makeatother

\newcommand{\A}{\mathbf{A}}

\newcommand{\G}{\mathbf{G}}

\renewcommand{\P}{\mathbf{P}}
\newcommand{\Q}{\mathbb{Q}}

\newcommand{\Z}{\mathbb{Z}}

\newcommand{\sP}{\mathcal{P}}

\newcommand{\Cor}{\operatorname{\mathbf{Cor}}}

\newcommand{\HI}{\operatorname{\mathbf{HI}}}

\newcommand{\Ext}{\operatorname{Ext}}
\newcommand{\ul}[1]{{\underline{#1}}}

\newcommand{\Cpx}{{\operatorname{\mathbf{Cpx}}}}

\newcommand{\NST}{\operatorname{\mathbf{NST}}}

\newcommand{\DM}{\operatorname{\mathbf{DM}}}

\newcommand{\Map}{\operatorname{map}}
\newcommand{\Hom}{\operatorname{Hom}}
\newcommand{\uHom}{\operatorname{\underline{Hom}}}
\newcommand{\uExt}{\operatorname{\underline{Ext}}}

\renewcommand{\Im}{\operatorname{Im}}

\newcommand{\Spec}{\operatorname{Spec}}

\newcommand{\Comp}{\operatorname{Comp}}
\newcommand{\Sm}{\operatorname{\mathbf{Sm}}}

\newcommand{\Sch}{\operatorname{\mathbf{Sch}}}
\newcommand{\Shv}{\operatorname{\mathbf{Shv}}}

\newcommand{\pro}[1]{\text{\rm pro}_{#1}\text{\rm--}}
\newcommand{\tr}{{\operatorname{tr}}}

\newcommand{\dlog}{{\operatorname{dlog}}}

\newcommand{\eff}{{\operatorname{eff}}}

\newcommand{\fin}{{\operatorname{fin}}}

\newcommand{\red}{{\operatorname{red}}}

\newcommand{\Zar}{{\operatorname{Zar}}}
\newcommand{\Nis}{{\operatorname{Nis}}}
\newcommand{\et}{{\operatorname{\acute{e}t}}}

\newcommand{\Res}{\operatorname{Res}}

\newcommand{\codim}{{\operatorname{codim}}}

\renewcommand{\lim}{\operatornamewithlimits{\varprojlim}}
\newcommand{\colim}{\operatornamewithlimits{\varinjlim}}
\newcommand{\holim}{\operatorname{holim}}

\newcommand{\ol}{\overline}

\renewcommand{\phi}{\varphi}
\renewcommand{\epsilon}{\varepsilon}
\renewcommand{\div}{\operatorname{div}}

\newcommand{\MNS}{\operatorname{\mathbf{MNS}}}
\newcommand{\MNST}{\operatorname{\mathbf{MNST}}}

\newcommand{\MSm}{\operatorname{\mathbf{MSm}}}

\newcommand{\MPST}{\operatorname{\mathbf{MPST}}}
\newcommand{\CI}{\operatorname{\mathbf{CI}}}

\newcommand{\Sq}{{\operatorname{\mathbf{Sq}}}}
\newcommand{\MSmsq}{{(\MSm)^{\Sq}}}

\newcommand{\bcube}{{\ol{\square}}}
\newcommand{\cube}{\square}

\newcommand{\M}{\mathbf{M}}

\newcommand{\ulMSm}{\operatorname{\mathbf{\underline{M}Sm}}}
\newcommand{\ulMNS}{\operatorname{\mathbf{\underline{M}NS}}}
\newcommand{\ulMPS}{\operatorname{\mathbf{\underline{M}PS}}}

\newcommand{\ulMNST}{\operatorname{\mathbf{\underline{M}NST}}}
\newcommand{\ulMCor}{\operatorname{\mathbf{\underline{M}Cor}}}

\newcounter{spec}
{\end{list}}%

\newtheorem{lemma}{Lemma}[section]
\newtheorem{thm}[lemma]{Theorem}

\newtheorem{prop}[lemma]{Proposition}

\newtheorem{cor}[lemma]{Corollary}

\theoremstyle{definition}
\newtheorem{defn}[lemma]{Definition}

\theoremstyle{remark}

\newtheorem{remark}[lemma]{Remark}

\newtheorem{nota}[lemma]{Notation}
\numberwithin{equation}{section}

\numberwithin{equation}{lemma}

\colorlet{LightRubineRed}{RubineRed!70!}

\def\lSm{\mathbf{lSm}}
\def\SmlSm{\mathbf{SmlSm}}
\def\Sm{\mathbf{Sm}}

\newcounter{elno}

\begin{document}

\def\aNis{a_{\Nis}}
\def\ulaNis{\underline{a}_{\Nis}}
\def\ulasNis{\underline{a}_{s,\Nis}}
\def\ulaNisfin{\underline{a}^{\fin}_{\Nis}}
\def\ulasNisfin{\underline{a}^{\fin}_{s,\Nis}}
\def\asNis{a_{s,\Nis}}
\def\ulasNis{\underline{a}_{s,\Nis}}
\def\qaq{\quad\text{ and }\quad}
\def\limcat#1{``\underset{#1}{\lim}"}
\def\Comp{\Comp^{\fin}}
\def\ulc{\ul{c}}
\def\ulb{\ul{b}}
\def\ulgam{\ul{\gamma}}
\def\MSm{\operatorname{\mathbf{MSm}}}
\def\MsigmaS{\operatorname{\mathbf{MsigmaS}}}
\def\ulMSm{\operatorname{\mathbf{\ul{M}Sm}}}
\def\ulMsigmaS{\operatorname{\mathbf{\ul{M}NS}}}

\def\ulMPS{\operatorname{\mathbf{\ul{M}PS}}}

\def\ulMsigmaS{\operatorname{\mathbf{\ul{M}PS}_\sigma}}
\def\ulMsigmaSTfin{\operatorname{\mathbf{\ul{M}PST}^{\fin}_\sigma}}
\def\ulMsigmaST{\operatorname{\mathbf{\ul{M}PST}_\sigma}}
\def\MsigmaS{\operatorname{\mathbf{MPS}_\sigma}}
\def\MsigmaST{\operatorname{\mathbf{MPST}_\sigma}}
\def\MsigmaSTfin{\operatorname{\mathbf{MPST}^{\fin}_\sigma}}

\def\ulMNS{\operatorname{\mathbf{\ul{M}NS}}}
\def\ulMNSTfin{\operatorname{\mathbf{\ul{M}NST}^{\fin}}}
\def\ulMNSfin{\operatorname{\mathbf{\ul{M}NS}^{\fin}}}
\def\ulMNST{\operatorname{\mathbf{\ul{M}NST}}}
\def\ulMEST{\operatorname{\mathbf{\ul{M}EST}}}
\def\MNS{\operatorname{\mathbf{MNS}}}
\def\MNST{\operatorname{\mathbf{MNST}}}
\def\MEST{\operatorname{\mathbf{MEST}}}
\def\MNSTfin{\operatorname{\mathbf{MNST}^{\fin}}}
\def\RSC{\operatorname{\mathbf{RSC}}}
\def\NST{\operatorname{\mathbf{NST}}}
\def\EST{\operatorname{\mathbf{EST}}}

\newcommand{\NS}{{\operatorname{\mathrm{NS}}}}

\def\LogRec{\operatorname{\mathbf{LogRec}}}
\def\Ch{\operatorname{\mathrm{Ch}}}

\def\MSmsq{\MSm^{\Sq}}
\def\Comp{\operatorname{\mathbf{Comp}}}
\def\uli{\ul{i}}
\def\ulis{\ul{i}_s}
\def\is{i_s}
\def\qfor{\text{ for }\;\;}
\def\CIlog{\operatorname{\mathbf{CI}}^{\mathrm{log}}}
\def\CIltr{\operatorname{\mathbf{CI}}^{\mathrm{ltr}}}
\def\CIt{\operatorname{\mathbf{CI}}^\tau}
\def\CItsp{\operatorname{\mathbf{CI}}^{\tau,sp}}
\def\ltr{\mathrm{ltr}}

\def\kX{\mathfrak{X}}
\def\kY{\mathfrak{Y}}
\def\kC{\mathfrak{C}}

\def\otCIsp{\otimes_{\CI}^{sp}}
\def\otCINissp{\otimes_{\CI}^{\Nis,sp}}

\def\hM#1{h_0^{\bcube}(#1)}
\def\hMNis#1{h_0^{\bcube}(#1)_{\Nis}}
\def\hMM#1{h^0_{\bcube}(#1)}
\def\hMw#1{h_0(#1)}
\def\hMwNis#1{h_0(#1)_{\Nis}}

\def\hetrec{h_{0, \et}^{\mathbf{rec}}}
\def\hetcube{h_{0, \et}^{\cube}}

\def\ihF#1{F^{#1}}
\def\ihFA{\ihF {\A^1}}

\def\istm{\iota_{st,m}}
\def\im{\iota_m}
\def\est{\epsilon_{st}}
\def\tL{\tilde{L}}
\def\tX{\tilde{X}}
\def\tY{\tilde{Y}}
\def\omegaCI{\omega^{\CI}}
\def\qwith{\;\text{ with} }
\def\aVNis{a^V_\Nis}
\def\ulMCorls{\ulMCor_{ls}}

\def\Zinf{Z_\infty}
\def\Einf{E_\infty}
\def\Xinf{X_\infty}
\def\Yinf{Y_\infty}
\def\Pinf{P_\infty}

\def\Lot{{\cubegm\otimes\cubegm}}

\def\Ln#1{\Lambda_n^{#1}}
\def\tLn#1{\widetilde{\Lambda_n^{#1}}}
\def\tild#1{\widetilde{#1}}
\def\otuCINis{\otimes_{\underline{\CI}_\Nis}}
\def\otCI{\otimes_{\CI}}
\def\otCINis{\otimes_{\CI}^{\Nis}}
\def\tF{\widetilde{F}}
\def\tG{\widetilde{G}}
\def\bcubered{\bcube^{\mathrm{red}}}
\def\cubegm{\bcube^{(1)}}
\def\cubegma{\bcube^{(a)}}
\def\cubegmb{\bcube^{(b)}}
\def\cubegmred{\bcube^{(1)}_{red}}
\def\cubegmreda{\bcube^{(a)}_{red}}
\def\cubegmredb{\bcube^{(b)}_{red}}

\def\LT{\bcube^{(1)}_{T}}
\def\LU{\bcube^{(1)}_{U}}
\def\LV{\bcube^{(1)}_{V}}
\def\LW{\bcube^{(1)}_{W}}
\def\LTred{\bcube^{(1)}_{T,red}}
\def\Lred{\bcube^{(1)}_{red}}
\def\LTred{\bcube^{(1)}_{T,red}}
\def\LUred{\bcube^{(1)}_{U,red}}
\def\LVred{\bcube^{(1)}_{V,red}}
\def\LWred{\bcube^{(1)}_{W,red}}
\def\PP{\P}
\def\AA{\A}

\def\LL{\bcube^{(2)}}
\def\LLred#1{\bcube^{(2)}_{#1,red}}
\def\LLredd{\bcube^{(2)}_{red}}
\def\Lredd#1{\bcube_{#1,red}}

\def\Lnredd#1{\bcube^{(#1)}_{red}}

\def\LLT{\bcube^{(2)}_T}
\def\LLTred{\bcube^{(2)}_{T,red}}

\def\LLU{\bcube^{(2)}_U}
\def\LLUred{\bcube^{(2)}_{U,red}}

\def\LLS{\bcube^{(2)}_S}
\def\LLSred{\bcube^{(2)}_{S,red}}
\def\tMCor{\Hom_{\MPST}}
\def\otHINis{\otimes_{\HI}^{\Nis}}

\def\Sh{\operatorname{\mathbf{Shv}}}
\def\Shv{\operatorname{\mathbf{Shv}}}
\def\PSh{\operatorname{\mathbf{PSh}}}
\def\Shltr{\operatorname{\mathbf{Shv}_{dNis}^{ltr}}}
\def\Shlog{\operatorname{\mathbf{Shv}_{dNis}^{log}}}
\def\Shvlog{\operatorname{\mathbf{Shv}^{log}}}
\def\SmlSm{\operatorname{\mathbf{SmlSm}}}
\def\lSm{\operatorname{\mathbf{lSm}}}
\def\lCor{\operatorname{\mathbf{lCor}}}
\def\SmlCor{\operatorname{\mathbf{SmlCor}}}
\def\PShltr{\operatorname{\mathbf{PSh}^{ltr}}}
\def\PShlog{\operatorname{\mathbf{PSh}^{log}}}
\def\lDM{\operatorname{\mathbf{logDM}^{eff}}}
\def\logDM{\operatorname{\mathbf{log}\mathcal{DM}^{eff}}}
\def\logDMlet{\operatorname{\mathbf{log}\mathcal{DM}^{eff}_{\mathrm{l\acute{e}t}}}}
\def\logDMone{\operatorname{\mathbf{log}\mathcal{DM}^{eff}_{\leq 1}}}
\def\logCI{\mathbf{logCI}} 

\def\DM{\operatorname{\mathbf{DM}^{eff}}}
\def\DMinf{\operatorname{\mathcal{DM}^{eff}}}
\def\lDA{\operatorname{\mathbf{logDA}^{eff}}}
\def\logDA{\operatorname{\mathbf{log}\mathcal{DA}^{eff}}}
\def\DA{\operatorname{\mathbf{DA}^{eff}}}
\def\Log{\operatorname{\mathcal{L}\textit{og}}}
\def\Rsc{\operatorname{\mathcal{R}\textit{sc}}}
\def\Pro{\mathrm{Pro}\textrm{-}}
\def\pro{\mathrm{pro}\textrm{-}}
\def\dg{\mathrm{dg}}
\def\plim{\mathrm{``lim"}}
\def\ker{\mathrm{ker}}
\def\coker{\mathrm{coker}}

\def\Alb{\operatorname{Alb}}
\def\bAlb{\mathbf{Alb}}
\def\Gal{\operatorname{Gal}}

\def\hofib{\mathrm{hofib}}
\def\triv{\mathrm{triv}}
\def\ABl{\mathcal{A}\textit{Bl}}
\def\divsm#1{{#1_\mathrm{div}^{\mathrm{Sm}}}}

\def\cA{\mathcal{A}}
\def\cB{\mathcal{B}}
\def\cC{\mathcal{C}}
\def\cD{\mathcal{D}}
\def\cE{\mathcal{E}}
\def\cI{\mathcal{I}}
\def\cS{\mathcal{S}}
\def\cM{\mathcal{M}}
\def\cO{\mathcal{O}}
\def\cP{\mathcal{P}}

\def\XP{X \backslash \sP}
\def\M0a{{}^t\cM_0^a}
\newcommand{\Ind}{{\operatorname{Ind}}}

\def\Xkbar{\overline{X}_{\overline{k}}}
\def\dx{{\rm d}x}

\newcommand{\dNis}{{\operatorname{dNis}}}
\newcommand{\loget}{{\operatorname{l\acute{e}t}}}
\newcommand{\ABNis}{{\operatorname{AB-Nis}}}
\newcommand{\sNis}{{\operatorname{sNis}}}
\newcommand{\sZar}{{\operatorname{sZar}}}
\newcommand{\set}{{\operatorname{s\acute{e}t}}}
\newcommand{\cofib}{\mathrm{Cofib}}

\newcommand{\Gmlog}{\G_m^{\log}}
\newcommand{\Gmlogred}{\overline{\G_m^{\log}}}

\newcommand{\varcolim}{\mathop{\mathrm{colim}}}
\newcommand{\varlim}{\mathop{\mathrm{lim}}}
\newcommand{\tensor}{\otimes}

\newcommand{\eq}[2]{\begin{equation}\label{#1}#2 \end{equation}}
\newcommand{\eqalign}[2]{\begin{equation}\label{#1}\begin{aligned}#2 \end{aligned}\end{equation}}

\def\varplim#1{\text{``}\varlim_{#1}\text{''}}
\def\det{\mathrm{d\acute{e}t}}

	\address{Institut für Mathematik, Universität Heidelberg, MATHEMATIKON INF 205, 69120  Heidelberg, Germany}
	\email{merici@mathi.uni-heidelberg.de}
	
	\thanks{This project was supported by the RCN project 313472 \emph{EMOHO - Equations in Motivic Homotopy} and the MSCA-PF project 101103309 \emph{MIPAC - Motivic Integral p-adic Cohomologies}}

	\title{A motivic integral p-adic cohomology}
	\author{Alberto Merici}
	
	\begin{abstract}
		We construct an integral $p$-adic cohomology that compares with rigid cohomology after inverting $p$. Our approach is based on the log-Witt differentials of Hyodo--Kato and log-\'etale motives of Binda--Park--{\O}stv{\ae}r. In case $k$ satisfies resolutions of singularities, we moreover prove that it agrees with the ``good'' integral $p$-adic cohomology of Ertl--Shiho--Sprang: from this we deduce some interesting motivic properties and a K\"unneth formula for the $p$-adic cohomology of Ertl--Shiho--Sprang.
	\end{abstract}
	
	\maketitle
	
	\tableofcontents
	\section{Introduction}
	
	Let $k$ be a perfect field of characteristic $p>0$ and let $W(k)$ be the ring of Witt vectors of $k$. It is now well established (see \cite{AbeCrew}) that there is \emph{no} cohomology theory\[
	R\Gamma\colon \Sm(k)\to \cD(W(k))^{op}
	\]
	satisfying the following three hypotheses:
	\begin{enumerate}[label=(\roman*)]
		\item\label{(i)} $R\Gamma(X)$ agrees with the rigid cohomology of $X$ after inverting $p$.
		\item\label{(ii)} The complex $R\Gamma(X)$ is bounded and the cohomology groups are finitely generated $W(k)$-modules.
		\item\label{(iii)} It satisfies finite \'etale descent, in the sense that for any \v Cech hypercover $X_\bullet \to X$ associated to a finite \'etale Galois cover $X_0\to X$, the induced morphism\[
		R\Gamma(X)\to R\Gamma(X_{\bullet})
		\]
		is an equivalence.
	\end{enumerate}
	In general the different incarnations of integral $p$-adic cohomology theories fail to satisfy \ref{(ii)} for $X$ that is not smooth and proper, as the $p$-torsion is very rich. In fact, as checked in \cite{AbeCrew}, if one assumes \ref{(i)} and \ref{(ii)}, the failure of \ref{(iii)} follows from a failure of descent along Artin--Schreier covers of $\A^1_k$; as these the archetype of wild ramification at $\infty$, this leads to think that if one relaxes \ref{(iii)} by avoiding wild ramification, there is a possibility of obtaining a positive result.

	In \cite{ertl2021integral}, with very strong assumptions on resolution of singularities\footnote{See Hypotheses 1.5-1.8 of \cite{ertl2021integral}: the assumption concerns strong and embedded resolutions of singularities and weak factorizations, analogous to \cite[Main Theorem I and II]{Hironaka} and \cite[0.0.1]{wlodarczyk}, or functorial resolutions as in \cite{abramobvich-temkin}}, the authors showed that the integral $p$-adic cohomology theory defined by sending $X\in \Sm(k)$ to the log de Rham--Witt cohomology of any smooth compactification $(\overline{X},\partial X)$ with log poles on the boundary $\partial X$, is well defined, functorial in $X$, and satisfies \ref{(i)} and \ref{(ii)} above, but not \ref{(iii)}.
	
	The goal of this paper is to prove the following result, without any assumption on resolutions of singularities of the base field $k$:
	\begin{thm}\label{thm:man-intro-absolute}
		There exists an integral $p$-adic cohomology that factors through Voevodsky's stable $\infty$-category of effective motives\[
		\begin{tikzcd}
			\Sm(k)\ar[rr,"R\Gamma_p"]\ar[dr] &&\cD(W(k))^{op}\\
			&\DMinf(k,\Z)\ar[ur].
		\end{tikzcd}
		\]
		For all $X\in \Sm(k)$, there is a canonical map to crystalline cohomology\[
		R\Gamma_p(X)\to R\Gamma_{\rm crys}(X)\]
		that factors through $R\Gamma_{\rm crys}(\ol{X},\partial X)$ whenever $X$ admits a smooth partial compactification $\ol{X}$. Moreover, after inverting $p$, there is a map to rigid cohomology: \[
		R\Gamma_p(X)[1/p]\to R\Gamma_{\rm rig}(X/K),\]
        where $K=W(k)[1/p]$.
	\end{thm}
	In particular the factorization through $\DMinf(k,\Z)$ implies the following descent property:
	\begin{enumerate}[label=(\roman*')]
		\setcounter{enumi}{2}
		\item\label{(iii)'} It satisfies Nisnevich descent, in the sense that for any \v Cech hypercover $X_\bullet \to X$ associated to a Nisnevich cover $X_0\to X$, the induced morphism\[
		R\Gamma_p(X)\to R\Gamma_p(X_{\bullet})
		\]
		is an equivalence.
	\end{enumerate}
	and the following properties:
	\begin{enumerate}
		\item Projective bundle formula: for $E\to X$ a vector bundle of rank $n+1$ and $\P(E)$ the associated projective bundle, there is an equivalence\[
		R\Gamma_p(\P(E))\simeq \bigoplus_{i=0}^{n} R\Gamma_p(X(i)[2i]).
		\]
		\item (Theorem \ref{thm:purity}) Purity: let $Z\subseteq X$ be a smooth closed subset of codimension $d$ and $U=X-Z$. Then there is a fiber sequence\[
		R\Gamma_p(Z(d)[2d])\to R\Gamma_p(X)\to R\Gamma_p(U).
		\]
	\end{enumerate}
	In case $k$ satisfies resolutions of singularities as in Notation \ref{nota:RS}\footnote{Notice that we assume strictly less hypotheses than \cite{ertl2021integral}: in particular, we do not assume embedded resolution and weak factorization of pairs as in Hypotheses 1.8 and 1.9}, we are able to prove that $R\Gamma_p(X)$ agrees with the log de Rham--Witt cohomology of any smooth compactification $(\overline{X},\partial X)$ with log poles on the boundary $\partial X$. In particular, we deduce (independently from \cite{ertl2021integral} and with milder assumptions on resolutions of singularities), the following:
	\begin{thm}\label{thm:main-intro}
		Assume that $k$ satisfies resolution of singularities as in \ref{nota:RS}. The map\[
		R\Gamma_p(X)\to R\Gamma_{\rm crys}((\ol{X},\partial X)/W(k))
		\] 
		is an equivalence for any smooth compactification. In particular the cohomology theory
		\begin{equation}\label{eq:coh-well-defined-intro}
			X \in \Sm(k) \mapsto R\Gamma_{\rm crys}((\ol{X},\partial X)/W(k))\in \cD(W(k)) 
		\end{equation}
		is well defined (i.e. does not depend on the choice of a compactification) and satisfies \ref{(i)} and \ref{(ii)}. 
	\end{thm}
	As observed in \cite[Proposition 2.24]{ertl2021integral}, the cohomology theory \eqref{eq:coh-well-defined-intro} does not satisfy \ref{(iii)}. On the other hand, it follows from our construction (see \ref{ssec:tame}) that it satisfies finite \emph{tame}\footnote{Thanks to \cite[Theorem 1.1]{KerzSchmidt}, there is no ambiguity in various notions of the adjective \emph{tame}.} descent in the following sense:
	\begin{enumerate}[label=(\roman*'')]
		\setcounter{enumi}{2}
		\item\label{(iii)''} For any \v Cech hypercover $X_\bullet \to X$ associated to a finite tame Galois cover $X_0\to X$, the induced morphism\[
		R\Gamma_p(X)\to R\Gamma_p(X_{\bullet})
		\]
		is an equivalence.
	\end{enumerate}
	Moreover, under these assumptions we can strengthen the motivic properties:
	\begin{enumerate}
		\item (Theorem \ref{thm:pbf}) Projective bundle formula: for $E\to X$ a vector bundle of rank $n+1$ and $\P(E)$ the associated projective bundle, there is an equivalence\[
		R\Gamma_p(\P(E))\simeq \bigoplus_{i=0}^{n} R\Gamma_p(X)[-i]
		\]
		which, if $X$ is proper, agrees with the projective bundle formula of \cite{Gros1985}.
		\item (Theorem \ref{thm:purity}) Purity: let $Z\subseteq X$ be smooth closed subset of codimension $d$ and $U=X-Z$. Then there is a fiber sequence\[
		R\Gamma_p(Z)[-d]\to R\Gamma_p(X)\to R\Gamma_p(U).
		\]
		\item (Theorem \ref{thm:kunneth}) K\"unneth: If $k$ satisfies \ref{nota:RS}, then for $X,Y\in \Sm(k)$ there is an equivalence \[
		R\Gamma_p(X)\widehat{\star^{L}} R\Gamma_p(Y)\simeq R\Gamma_p(X\times Y)
		\]
		where $\widehat{\star^{L}}$ is Ekhedal's complete-derived product of modules on the Raynaud ring $R(k)$. 
	\end{enumerate}
	We give an overview of our arguments: recall the log motivic categories  $\mathbf{log}\mathcal{DA}^{\mathrm{eff}}_{\mathrm{l\et}}$ and $\mathbf{log}\mathcal{DM}^{\mathrm{eff}}_{\mathrm{l\et}}$ from \cite{BPO} (see \ref{ssec:logDM}). The main technical result is the following:
	\begin{thm}\label{thm:intro-crys-S}
		For all $S\in \Sm(k)$, the cohomology of the log de Rham--Witt sheaves $W_m\Lambda^q_{-/S}$ is representable in the log motivic category $\mathbf{log}\mathcal{DA}^{\mathrm{eff}}_{\mathrm{l\et}}(S,\Z)$. Moreover, if $S=\Spec(k)$, then the sheaves $W_m\Lambda^q$ have log transfers and their cohomology is representable in $\mathbf{log}\mathcal{DM}^{\mathrm{eff}}_{\mathrm{l\et}}(k,\Z)$
	\end{thm}
	Once we have this result at our disposal, we deduce that for all $S\in \Sm(k)$, the log crystalline cohomology of \cite{Hyodo}:\[
	R\Gamma_{\rm crys}(-/W(k))\colon \lSm(S)\to \cD(W(k))^{op} \qquad  X \mapsto \holim_m R\Gamma(X,W_m\Lambda^\bullet).
	\]
	factors through $\mathbf{log}\mathcal{DA}^{\mathrm{eff}}_{\mathrm{l\et}}(S,\Z)$ (resp. $\mathbf{log}\mathcal{DM}^{\mathrm{eff}}_{\mathrm{l\et}}(k,\Z)$ if $S=\Spec(k)$). Here $\lSm(S)$ is the category of fine and saturated log schemes that are log smooth and separated over $S$ equipped with the trivial log structure (see \cite{BPO}).
	
	Now, by \cite[Proposition 2.5.7]{DoosungA1log}, there is a fully faithful functor\[
	\mathcal{DM}^{\mathrm{eff}}(k,\Z)\xrightarrow{\omega^*} \mathrm{\log}\mathcal{DM}^{\mathrm{eff}}(k,\Z)
	\]
	from the category of Voevodsky motives to the category of log motives. We then define $R\Gamma_p$ as the composition
	\begin{equation}\label{eq:RGammap}
		\Sm(k)\xrightarrow{M^{V}} \mathcal{DM}^{\mathrm{eff}}(k,\Z)\xrightarrow{\omega^*} \mathrm{\log}\mathcal{DM}^{\mathrm{eff}}(k,\Z)\xrightarrow{L_{\loget}} \mathrm{\log}\mathcal{DM}^{\mathrm{eff}}_{\loget}(k,\Z) \xrightarrow{R\Gamma_{\rm crys}} \cD(W(k))^{op}
	\end{equation}
	The canonical map $R\Gamma_p(X)\to R\Gamma_{\rm crys}(X)$ is induced by the $\A^1$-localization functor $M(X,\triv)\to \omega^*M^{\A^1}(X)$, and the canonical factorization through $R\Gamma_{\rm crys}(\ol{X},\partial X)$ as in Theorem \ref{thm:man-intro-absolute} is induced by the factorization \[
    M(X)\to M(\ol{X},\partial X)\to \omega^*M^{\A^1}(X).
    \]
	The key technical point of the proof of Theorem \ref{thm:intro-crys-S} is the transfer structure that is needed in order to exploit  \cite[Theorem 8.2.16]{BPO}. This will follow from a comparison with the sheaf with transfers $\Log(W_m\Omega^n)$, where the functor $\Log$ was constructed in \cite{shujilog} for the reciprocity sheaves of \cite{KSY}: the proof will follow from a comparison of two Gysin sequences, after some general properties on $\bcube$-invariant $\dNis$-sheaves which generalize the work of Morel \cite{Morel} on $\A^1$-invariant sheaves.
	
	If moreover we assume that $k$ satisfies strong resolutions of singularities as in Notation \ref{nota:RS}, by \cite[Theorem 8.2.16]{BPO} there is an equivalence
	\begin{equation}\label{eq:comp-intro}
		\omega^*M^{\A^1}(X)\simeq M(\ol{X},\partial X),
	\end{equation}
	where $\ol{X}$ is \emph{any} smooth Cartier compactification of $X$ with $\ol{X}-X$ a simple normal crossing divisor that supports the log structure $\partial X$. This implies that there is a canonical equivalence:\[
	R\Gamma_p(X)\simeq R\Gamma_{\rm cris}(\ol{X},\partial X)
	\]
	that does \emph{not} depend on the choice of $(\ol{X},\partial X)$, and the canonical map $R\Gamma_p(X)\to R\Gamma_{\rm cris}(X)$ agrees with the map $R\Gamma_{\rm cris}(\ol{X},\partial X)\to R\Gamma_{\rm cris}(X)$.
	
	In particular, we conclude that \eqref{eq:coh-well-defined-intro} is well defined and satisfies \ref{(i)} and \ref{(ii)}. Finally, the property \ref{(iii)''} follows from a comparison between tame and log-\'etale coverings due to Fujiwara and Kato (see Lemma \ref{lem:tame-vs-log}), and the refined motivic properties follow from the fact that \eqref{eq:comp-intro} implies that the functor $\omega^*$ is monoidal.
	
	We remark that resolutions of singularity is only used in the explicit computation of the object $\omega^*M^{\A^1}(X)$ in $\logDM(k,\Z)$. We believe that in fact one can prove that $R\Gamma_p$ satisfies \ref{(i)}, \ref{(ii)} and \ref{(iii)''} without any assumption on resolution of singularities, provided that one computes $\omega^*M^{\A^1}(X)$. This is a future work in progress, which we believe to be linked with a suitable definition of a \emph{tame} motivic homotopy type of $X$, where tame is in the sense of \cite{HS2020}. Moreover, we remark that alterations are probably not useful for this computation, as we are not inverting $p$.
	
	\subsection*{Acknowledgements}
	The author would like to thank J. Ayoub for great advice and feedback about some delicate points of Section \ref{sec:complements}, for pointing out some mistakes in previous versions of the paper and for suggesting the use of residue maps. He also thanks F. Binda, V. Ertl, D. Park, K. R\"ulling, S. Saito and P.A. {\O}stv{\ae}r for many valuable discussions and comments, together with the anonymous referees for their meticulous analysis of the paper, providing helpful comments which filled some gaps in the arguments and led to an improved presentation. This project started while the author was a visiting fellow at the University of Milan, the main results were obtained while the author was a postdoc supported by the RCN project \emph{EMOHO} at the University of Oslo, and the final version was settled with he support of the MSCA-PF \emph{MIPAC} carried out at the University of Milan. The author is very thankful for the hospitality and the great work environment. 
	\begin{nota}
 For space constrains reasons, if the base field or the coefficients are evident or not important, we will often omit them (e.g. we will write $\Sh_{\Nis}^{\tr}$ instead of $\Sh_{\Nis}^{\tr}(k,A)$ for a base field $k$ and a ring $A$ that are evident from the context).
 \end{nota}
	\begin{nota}\label{nota:big-to-small}
		For $X\in \Sm(k)$, we let $X_{\Zar}$ (resp $X_{\Nis}$) be the small Zariski (resp. Nisnevich) site of $X$. 
		
		For $F\in \PSh(k)$, we let $F_{X}$ be the presheaf on $X_{\Nis}$ such that for any \'etale map $U\to X$,\[
		F_{X}(U) = F(U).
		\]
		If $F\in \Sh_{\Zar}(k)$ (resp. $\Sh_{\Nis}(k)$), then by definition $F_X \in \Sh(X_{\Zar})$ (resp. $ \Sh(X_{\Nis})$).
		Similarly, for $X=(\ul{X},\partial X)\in \lSm(k)$ and $F\in \PSh^{\log}(k)$, we let $F_{X}$ be the presheaf on $\ul{X}$ such that for any \'etale map $U\to \ul{X}$, \[
		F_{X}(U) = F(U,\partial X_{|U}).
		\]
		If $F\in \Sh_{\sZar}^{\log}(k)$ (resp. $\Sh_{\sNis}^{\log}(k)$ or $\Sh_{\dNis}^{\log}(k)$), then by definition $F_X \in \Sh(X_{\Zar})$ (resp. $ \Sh(X_{\Nis})$).
	\end{nota}
	
	\section{Recollections}
	
	In this section, we recall the main results on logarithmic motives and reciprocity sheaves.

	\subsection{Recollections on log motives}\label{ssec:logDM} 
	We recall the construction of the $\infty$-category of logarithmic motives of \cite{BPO} and some properties. The standard reference for log schemes is \cite{ogu}. 
	We denote by ${\lSm}(S)$ the category of fs log smooth log schemes over the log scheme $(S,\triv)$. We are typically interested in the case where $S = \Spec(k)$. For $X\in {\lSm}(S)$, we will denote by $\ul{X}$ the underlying scheme of $X$, by $X^{\circ}$ the open subscheme where the log structure is trivial (we will refer to it as the \emph{trivial locus}), by $\partial X$ the log structure and by $|\partial X|$ its support, seen as a reduced closed subscheme of $X$. 
	
	Let $\SmlSm(S)$ be the full subcategory of $\lSm(S)$ having for objects $X\in \lSm(S)$ such that $\ul{X}$ is smooth over $S$. By e.g. \cite[Lemma A.5.10]{BPO}, then in this case $\partial X$ is supported on a strict normal crossing divisor on $\ul{X}$ and the log scheme $(\ul{X},\partial X)$  is isomorphic to the compactifying log structure associated to the open embedding $\ul{X}-|\partial X| \hookrightarrow \ul{X}$. If $D$ is a strict normal crossing divisor on $\ul{X}$, we will often write $(\ul{X},D)\in \SmlSm(S)$ meaning the log scheme with log structure supported on $D$.
	
	A morphism $f\colon X\to Y$ of fs log schemes is called \emph{strict} if the log structure on $X$ is the pullback log structure from $Y$. In case both $X$ and $Y$ are objects of $\SmlSm(S)$, this translate to an equality $\partial X = f^*(\partial Y)$ as reduced normal crossing divisors on $X$. If $Z$ is a closed subscheme of $\ul{X}$, we will often denote by $X-Z\subseteq X$ the strict open immersion $(\ul{X}-Z,\partial X_{|\ul{X}-Z})\hookrightarrow X$.
	
	We denote by $\PShlog(S, A)$ the category of presheaves of $A$-modules on $\lSm(S)$. It has naturally the structure of closed monoidal category. If $\tau$ is a Grothendieck topology on $\lSm(S)$ (see below), we write $\mathbf{Shv}^{\rm log}_\tau(S, A)$ for the full subcategory of $\PShlog(S, A)$ consisting of $\tau$-sheaves. 
	
	Let $\widetilde{\SmlSm}(S)$ be the category of fs log smooth $S$-schemes $(X,\partial X)$ which are essentially smooth over $S$, i.e. a limit $\lim_{i \in I} X_i$ over a cofiltered set $I$, 
	where $X_i \in \SmlSm(S)$ and all transition maps are strict \'etale (i.e. they are strict maps of log schemes such that the underlying maps $f_{ij}\colon \ul{X_i}\to \ul{X_j}$ are \'etale). For $X\in \SmlSm(S)$ and $x\in X$, let $\iota\colon\Spec(\mathcal{O}_{X,x})\to X$ and $\iota^h\colon\Spec(\mathcal{O}_{X,x}^h)\to X$ be the canonical morphism. Then we will denote by $X_x$ and $X_x^h$ respectively the localization $(\Spec(\mathcal{O}_{X,x}),\iota^*(\partial X))$ and henselization $(\Spec(\mathcal{O}_{X,x}^h),(\iota^h)^*(\partial X))$: both lie in $\widetilde{\SmlSm}(S)$.
	We frequently allow $F\in \PShlog(S,A)$ to take values on objects of $\widetilde{\SmlSm}(S)$ by setting
	$F(X) := \colim_{i \in I} F(X_i)$ for $X$ as above.

	For $\tau$ a Grothendieck topology on $\Sch(S)$, the $\emph{strict}$ topology $s\tau$ on $\SmlSm(S)$ is the Grothendieck topology generated by covers $\{e_i\colon X_i\to X\}$ such that $\underline{e_i}\colon \underline{X_i}\to \underline{X}$ is a $\tau$-cover and each $e_i$ is strict. 
	Moreover, recall that a morphism of fs log schemes $f : X \to Y$ is called Kummer-\'etale if it is exact and log \'etale, see \cite[Proposition A.8.4]{BPO}. A typical example is given by $(\A^1_k,0) \to (\A^1_k,0)$, $t\mapsto t^n$, where $n$ is coprime to $\mathrm{char}(k)$. The Kummer-\'etale topology is the topology generated by Kummer-\'etale covers: it is finer than the strict \'etale topology.
	Recall from \cite[3.1.4]{BPO} that a cartesian square of fs log schemes
	\[
	\begin{tikzcd}
		Y' \ar[r, "g'"]\ar[d, "f'"] & Y \ar[d, "f"]\\
		X' \ar[r, "g"] & X
	\end{tikzcd}
	\]
	is a \emph{dividing distinguished square} if $Y'=X'=\emptyset$ and $f$ is a surjective proper log \'etale monomorphism (see \cite[Remark 2.2.4]{BPO-SH} for more details on log modifications). The collection of dividing distinguished squares forms a cd structure on $\SmlSm(S)$, called the \emph{dividing cd structure}. For $\tau$ a Grothendieck topology on $\Sch(S)$, the $\emph{dividing}$ topology $d\tau$ on $\SmlSm(k)$ is the topology on $\SmlSm(k)$ generated by the strict topology $s\tau$ and the dividing cd structure. Moreover, we denote by $\loget$ the \emph{log-\'etale topology}, i.e. topology generated by the Kummer-\'etale topology and the dividing cd structure. We denote by $\mathbf{Shv}_{d\tau}^{\rm log}(S, A)\subset \PShlog(S, A)$  (resp. $\mathbf{Shv}_{\mathrm{\mathrm{k\et}}}^{\rm log}(S, A)$, resp.  $\mathbf{Shv}_{\mathrm{\loget}}^{\rm log}(S, A)$) the subcategory of $d\tau$-sheaves (resp. $\mathrm{k\et}$-sheaves, resp. $\loget$-sheaves), with exact sheafification functors $a_{d\tau}$, (resp. $a_{\mathrm{k\et}}$, resp. $a_{\loget}$). By \cite[Theorem 1.2.2]{BPO}, for $\tau\in \{\Nis,\et\}$, $F$ an $s\tau$-sheaf (resp. a $\mathrm{k\et}$-sheaf) we have\[
	H^n_{d\tau}(X,a_{d\tau}F)\simeq \colim_{Y} H^n_{s\tau}(Y,F)\quad (\textrm{resp.} H^n_{\loget}(X,a_{\loget}F)\simeq \colim_{Y} H^n_{\mathrm{k\et}}(Y,F))
	\]
	where the colimit runs over all log modifications $Y\to X$.
	Following \cite{BPO}, we denote by $\lCor(k)$ the category of finite log correspondences over $k$. It is a variant of the Suslin--Voevodsky category of finite correspondences $\Cor(k)$. It has the same objects as $\SmlSm(k)$\footnote{Notice that this notation conflicts with the notation of \cite{BPO} where the objects were the same as $\lSm(k)$, although the categories of sheaves are the same in light of \cite[Lemma 4.7.2]{BPO}}, and morphisms are given by the free abelian subgroup
	\[ \lCor(X,Y) \subseteq  \Cor(X- \partial X, Y- \partial Y)\]
	generated by elementary correspondences  $V^o\subset (X- \partial X) \times (Y- \partial Y)$ such that the closure $V\subset \ul{X}\times \ul{Y}$ is finite and surjective over (a component of) $\ul{X}$ and such that there exists a morphism of log schemes $V^N \to Y$, where $V^N$ is the fs log scheme whose underlying scheme is the normalization of $V$ and whose log structure is given by the inverse image log structure along the composition $\underline{V^N} \to \ul{X}\times \ul{Y} \to \ul{X}$. See \cite[2.1]{BPO} for more details, and for the proof that this definition gives indeed a category. 
	
	Additive presheaves (of $A$-modules) on the category $\lCor(k)$ will be called \emph{presheaves (of $A$-modules) with log transfers}. Write $\PShltr(k, A)$ for the resulting category, and by $\mathbf{Shv}_{d\tau}^{\rm ltr}(k, A)\subset \PShltr(k, A)$  (resp.  $\mathbf{Shv}_{\mathrm{\loget}}^{\rm ltr}(k, A)$) the subcategory of $d\tau$ (resp. $\loget$) sheaves with transfers. By \cite[Theorem 4.5.7]{BPO}, the sheafification $a_{d\tau}$ (resp $a_{\loget}$) preserves transfers.
	
	Let $T\in\{\dNis, \mathrm{d\et},\loget\}$. 
	Let $\cD(\mathbf{Shv}_{T}^{\rm log}(S, A))$, resp. $\cD(\mathbf{Shv}_{T}^{\rm ltr}(k, A))$, be the derived stable $\infty$-category of the Grothendieck abelian category $\mathbf{Shv}_{T}^{\rm log}(S, A)$, resp. $\mathbf{Shv}_{T}^{\rm ltr}(k, A)$, as in \cite[Section 1.3.5]{HA}: it is equivalent to the underlying $\infty$-category of the model category $\Cpx(\PShlog(S, A))$, resp. $\Cpx(\PSh^{\rm ltr}(k, A))$, with the $T$-local model structure used in \cite{BPO} and \cite{BindaMerici}. 
	
	Finally (see \cite[Section 5.2]{BPO}), let $\bcube:=(\P^1,\infty)$.
	\begin{defn} The stable $\infty$-category $\mathbf{log}\mathcal{DA}^{\textrm{eff}}_T(S,A)$ (resp. $\mathbf{log}\mathcal{DM}^{\textrm{eff}}_T(k,A)$) is the localization of the stable $\infty$-category $\cD(\mathbf{Shv}_{T}^{\rm log}(S, A))$ (resp. $\cD(\mathbf{Shv}_{T}^{\rm ltr}(k, A))$) with respect to the class of maps
		\[
		(a_{T}A(\bcube\times X))[n]\to (a_{T}A(X))[n]\quad(\textrm{resp. } (a_{T}A_{\ltr}(\bcube\times X))[n]\to (a_{T}A_{\ltr}(X))[n])
		\]
		for all $X\in \SmlSm(k)$ and $n\in \Z$. 
		We let $L^{\log}_{(T,\bcube)}$ (resp. $L^{\ltr}_{(T,\bcube)}$)
		be the localization functor and for $X\in \SmlSm(k)$, we will let $M_T^{\log}(X)=L^{\log}_{(T,\bcube)}(a_TA(X))$ (resp. $M_T^{\ltr}(X) = L^{\ltr}_{(T,\bcube)}(a_TA_{\ltr}(X))$).
	\end{defn}
	
	If $T=\dNis$, we will often drop it from the notation. We recall the following result \cite[Theorem 5.7]{BindaMerici}:
	\begin{thm}\label{thm:t-structure}
		The standard $t$-structures of $\cD(\Shv^{\rm log}_{\dNis}(k, A))$ and $\cD(\Shv^{\rm ltr}_{\dNis}(k, A))$ induce accessible $t$-structures on $\mathbf{log}\mathcal{DA}^{\mathrm{eff}}(k,A)$ and $\logDM(k,A)$ compatible with filtered colimits in the sense of \cite[Definition 1.3.5.20]{HA}, called the \emph{homotopy $t$-structures.}
	\end{thm}
	
	We denote by $\mathbf{logCI}(k,A)$ (resp. $\mathbf{logCI}^{\ltr}(k,A)$) its heart, which is then identified with the category of strictly $\bcube$-invariant $\dNis$-sheaves (resp. $\dNis$-sheaves with log
    transfers) and it is a Grothendieck abelian category. The inclusions
	\[
	\mathbf{logCI}(k,A)\hookrightarrow \Shv^{\rm log}_{\dNis}(k, A)\quad 
	\mathbf{logCI}^{\ltr}(k,A)\hookrightarrow \Shv^{\rm ltr}_{\dNis}(k, A)
	\]
	admit left adjoints $h_0^{\bcube}, h_0^{\bcube,\ltr}$, both given by the same formula $F\mapsto \pi_0(L_{(d\tau,\bcube)}(F[0]))$, and right adjoints $h^0_{\bcube}, h^0_{\bcube,\ltr}$  (see \cite[Proposition 5.8]{BindaMerici}), in particular this inclusions are exact.
	
	Recall from  \cite[(4.3.4)]{BPO} that the functor $\omega:X\mapsto X^{\circ}$ induces an adjunction
	\begin{equation}\label{omegaadjunction}
		\begin{tikzcd}
			\Shv_{d\tau}^{\log}(k,A)\arrow[rr,shift left=1.5ex,"\omega_\sharp"]&& \Shv_{\tau}(k,A)\arrow[ll,"\omega^*"]
		\end{tikzcd}
	\end{equation} 
	where for $Y\in \Sm(k)$, $\omega_{\sharp} F(Y) = F(Y,\textrm{triv})$ and for $X\in \Sm(k)$, $\omega^*F(X)=F(\underline{X}-|\partial X|)$. Moreover, since $\omega$ is monoidal by construction, $\omega_\sharp$ is monoidal. It is immediate that $\omega^*$ and $\omega_\sharp$ preserve transfers, and induce adjunctions of stable $\infty$-categories:\[
	\begin{tikzcd}
		\cD(\Shv_{d\tau}^{\log}(k,A))\arrow[r,shift left=1.5ex,"\omega_\sharp"]& \cD(\Shv_{\tau}(k,A))\arrow[l,"\omega^*"] &&\cD(\Shv_{d\tau}^{\ltr}(k,A))\arrow[r,shift left=1.5ex,"\omega_\sharp"]& \cD(\Shv_{\tau}^{\tr}(k,A))\arrow[l,"\omega^*"] 
	\end{tikzcd}
	\]
	As observed in \cite[Proposition 5.11]{BindaMerici}, the adjunctions above induce:\[
	\begin{tikzcd}
		\mathbf{log}\mathcal{DA}^{\rm eff}(k,A)\arrow[r,shift left=1.5ex,"L\omega_\sharp"]& \mathcal{DA}^{\rm eff}(k,A) \arrow[l,"\omega^*"] &&\logDM(k,A)\arrow[r,shift left=1.5ex,"L^{\ltr}\omega_\sharp"]& \mathcal{DM}^{\rm eff}(k,A),\arrow[l,"\omega^*"] 
	\end{tikzcd}
	\]
	and the functors $\omega^*$ are $t$-exact for the homotopy $t$-structures. By \cite[Proposition 2.5.7]{DoosungA1log}, both $\omega^*$ are fully faithful and their essential images agree with the $(\A^1,\triv)$-local categories.
	\subsection{Recollection on reciprocity sheaves}
	
	Recall the abelian category of reciprocity sheaves $\RSC_{\Nis}(k)\subseteq \Sh_{\Nis}^{\rm tr}(k)$, defined in equivalent ways in \cite{KSY}, \cite{KSY-RecII} and \cite{RulSaito}.  It contains Voevodsky's category $\HI^{\tr}$of $\A^1$-invariant sheaves and many other interesting non-$\A^1$-invariant sheaves like $\Omega^n$ and $W_m\Omega^n$.
	
	By \cite{shujilog} there is a fully faithful and exact functor \[
	\Log\colon \RSC_{\Nis}(k,A)\to \logCI^{\ltr}(k,A)
	\]
	such that if $F\in \HI^{\tr}$ then $\Log(F)=\omega^*F$.
	We recall the construction of the Gysin map in \cite{BindaRuellingSaito}, in the special case that will be needed later. Let $F\in \RSC_{\Nis}(k)$: we put $\gamma^1F := \uHom_{\RSC_{\Nis}}(\G_m,F)$. Let $X\in \Sm(k)$  and $i\colon D\subseteq X$ a smooth divisor. By \cite[Theorem 5.10]{BindaRuellingSaito} there is a cycle class map of sheaves on the small Nisnevich site of $X$ (see Notation \ref{nota:big-to-small}):
	\begin{equation}\label{eq:BRS-cycle-class}
		c(D)\colon \gamma^1 F_{X}\to R^1\ul{\Gamma}_D (F_{X}),
	\end{equation}
	which can be described as follows. Let $j\colon X-D\to X$ be the open immersion so that $R^1\ul{\Gamma}_D (F_{X})\cong j_*F_{X-D}/F_{X}$. If $t$ is a local equation of $D$ in $X$, then let $f_t\colon X-D \to \A^1-\{0\}$ be the induced morphism and let $\Delta\colon X-D \to X\times(X-D)$ be the composition of the diagonal and $j\times id$. Let $a\in\gamma^1F(X)$ that restricts to $\tilde a\in F(X\times (\A^1-\{0\}))$ via the surjective map $\Z_{\tr}(\A^1-\{0\})\to \G_m$. Then by \cite[(5.10.3)]{BindaRuellingSaito} we have that $c(D)(a)$ is the class of $\Delta^*(id\times \Gamma_t)^*(\tilde a)$ in $j_*F_{X-D}/F_{X}$.
	
	Assume now that $i\colon D\to X$ has a retraction $q\colon X\to D$. Then we have a map:
	\begin{equation}\label{eq:BRS-local-Gysin}
		i_*\gamma^1F_D \to i_*q_* \gamma^1F_X \xrightarrow{i_*q_*\eqref{eq:BRS-cycle-class}} i_*q_* R^1\ul{\Gamma}_D (F_{X})= i_*q_*i_*R^1i^! (F_{X}) \cong i_*R^1i^! (F_{X}) = R^1\ul{\Gamma}_D (F_{X}).
	\end{equation}
	By composing with the ``forget support" map and by shifting $-1$ one gets a map in the homotopy category $D(\Sh_{\Nis}(X))$:
	\begin{equation}\label{eq:BRS-Gysin}
		g^{\rm{BRS}}_{D/X}\colon i_*\gamma^1F_D[-1] \to F_{X}
	\end{equation}
	which by \cite[Theorem 7.16]{BindaRuellingSaito} corresponds to the extension
	\begin{equation}\label{eq:BRS-Gysin-ext}
		0\to  F_{X} \to \Log(F)_{(X,D)}\to i_*\gamma^1F_D \to 0 .
	\end{equation}

	\subsection{Recollection on log de Rham--Witt sheaves and complex}\label{ssec:logdRW}

	Let $k$ be a field of characteristic $p>0$ and let $S$ be a Noetherian $k$-scheme and $W_m(S)$ the sheaf of $m$-truncated Witt vectors over $S$, with Frobenius $\sigma$. 
	
	For $X\in \SmlSm(S)$ and $m,q\geq 0$, we consider $W_m\Lambda^q_{X/(S,\triv)}$ the sheaf on $\ul{X}_{\Zar}$ of logarithmic differentials of \cite{matsuue}. For $\partial X$ trivial, it agrees with $W_m\Omega^q_{\ul{X}/S}$ of \cite{IllusiedRW}, and if $q=0$ it agrees with the $m$-truncated Witt sheaf of $\ul{X}$ (in particular, it does not depend on $\partial X$). The assignment\[
	X\mapsto \Gamma(X,W_m\Lambda^q_{X/(S,\triv)})
	\]
	defines a $\set$-sheaf on $\SmlSm(S)$ by \cite[Proposition 3.7]{matsuue}, that we will briefly indicate as $W_m\Lambda^q_S$. 
	It comes equipped with usual maps $F$, $V$ and $d$, which make $\{W_m\Lambda^{\bullet}_S\}_{m}$ a universal log $F$-$V$ procomplex. Notice that in \cite[7.2]{matsuue} (and only in that section), this sheaf is indicated as $W_m\widetilde{\Lambda}^q$, and $W_m\Lambda^q$ is the sheaf over a base with log structure. We will only work over a base with trivial log structure, so the extra notation $W_m\widetilde{\Lambda}^q$ is unnecessary.
	
	Recall the weight filtration on $W_m\Lambda^q$:\[
	P_jW_m\Lambda^q_{X} :=\Im(W_m\Lambda^j_{X}\otimes W_m\Omega^{q-j}_{\ul{X}}\to W_m\Lambda^q_{X})
	\]
	For $X\in \SmlSm(S)$, let $\partial X^{(j)}$ be the disjoint union of the intersections of the $j$ components of $|\partial X|$ and let $\iota^{(j)}\colon \partial X^{(j)}\to \ul{X}$ be the inclusion. By convention, we put $\partial X^{(0)}:=\ul{X}$. 
	For all $j$, we have exact sequences of sheaves on $\ul{X}_{\Nis}$ (see \cite[Lemma 8.4]{matsuue} and \cite[1.4]{mokrane}):
	\begin{equation}\label{eq:weight-es}
		0\to P_{j-1}W_m\Lambda^q_{X}\to P_{j}W_m\Lambda^q_{X}\xrightarrow{{\rm Res}} \iota^{(j)}_*W_m\Omega^{q-j}_{\partial X^{(j)}}\to 0
	\end{equation}
	In case $\partial X$ has only one component $D$, then the sequence above gives an exact sequence of sheaves on $\ul{X}_{\Nis}$
	\begin{equation}\label{eq:weight-es-sm}
		0\to W_m\Omega^q_{\ul{X}} \to W_m\Lambda^n_{X} \to \iota^{(1)}_*W_m\Omega^{q-1}_{D} \to 0.
	\end{equation}
	The class of this extension gives a map in  the homotopy derived category $D(\Sh(\ul{X}_{\Nis}))$:
	\begin{equation}\label{eq:Gros-Gysin}
		g^{\rm{Gros}}_{D/X}\colon i_*W_m\Omega^{q-1}_{D}[-1]\to  W_m\Omega^q_{\ul{X}}.
	\end{equation}
	By \cite[4.6]{mokrane}, this map agrees with the Gysin map of \cite[II, Prop 3.3.9]{Gros1985}: if $D$ is locally given by the equation $t=0$, then for any local lift $D_m\hookrightarrow X_m$ to $W_m$, for any $[\omega]\in W_m\Omega^{q-1}_{D}$ and any lift $\omega\in \Omega^{q-1}_{D_m/W_m}$, then $g^{\rm{Gros}}_{D/X}$ maps $[\omega]$ to the image of $\omega\wedge \dlog(t) \in \ul{H}^1_D(\Omega^q_{X_m/W_m})$. By composing with the ``forget support" map and by shifting $-1$ one gets $\eqref{eq:Gros-Gysin}$.
	
	In general, we consider the cohomology theory\[
	R\Gamma(-,W\Lambda_S)\colon \SmlSm(S) \to \cD(W(k))^{op}\qquad X\mapsto \holim_m R\Gamma(X,W_m\Lambda_{X/S}^{\bullet})
	\]
	Recall that the Raynaud ring $R(k)$ of Ekhedal \cite{EkedahlII} is the graded non-commutative $W(k)$-algebra $R^0\oplus R^1$ generated by elements $F$ and $V$ in degree $0$ and $d$ in degree $1$,
	subject to the relations $FV = VF =p$, $Fa = a^\sigma F$, $aV = Va^{\sigma}$ , $ad = da$ $(a\in W(k))$, $d^2 =0$, $FdV =d$. Recall that the category of (left-)modules over the Raynaud ring admits a universal tensor product $\star$ and a derived-complete tensor product $\widehat{\star^L}$ in $\cD(R(k))$. We recall that similarly to \cite[(2.6.1.2) and (5.1.1)]{Illusiefinite}, the dga structure of the log de Rham--Witt complex and the \v Cech comparison give rise to a product map in $\cD(W(k)[d])$\[
	R\Gamma(X,W\Lambda_{X/S})\otimes^{L} R\Gamma(Y,W\Lambda_{X/S}) \to R\Gamma(X\times Y,W\Lambda_{X/S})
	\]
	which is functorial in $X$ and $Y$. Since $R\Gamma(X\times Y,W\Lambda)\in \cD(R(k))$ and is derived-complete by construction, the product map induces a map
	\begin{equation}\label{eq:kunneth}
		R\Gamma(X,W\Lambda)\widehat{\star^{L}} R\Gamma(Y,W\Lambda) \to R\Gamma(X\times Y,W\Lambda)
	\end{equation}
	which is functorial in $X$ and $Y$ by \cite[I. Theorem 6.2]{EkedahlII}. If $X$ and $Y$ have trivial log structure and $S=\Spec(k)$, the previous map is an equivalence by \cite[Corollary 1.2.5]{EkedahlII}\footnote{a small descent argument is needed, see \cite[Theorem 5.1.2]{Illusiefinite}}. We will prove in Proposition \ref{prop:kunneth} that \eqref{eq:kunneth} is always an equivalence using the motivic properties of the sheaves $W_m\Lambda^n$. 
	
	\section{Complements on \texorpdfstring{$\bcube$}{box}-local sheaves}\label{sec:complements}
	
	In this section, we prove some technical results on $\bcube$-local sheaves, generalizing the results of \cite{Morel} on $\A^1$-local sheaves. First we show the following immediate result:
	\begin{lemma}\label{lem:cohomology-open-P1}
		Let $U\subseteq \A^1$ be dense open. Then for all $F\in \logCI(k,A)$ and $i>0$\[
		H^i_{\dNis}(U,F) = 0
		\]
		\begin{proof}
            For $i>1$, this follows by cohomological dimension of the $\dNis$ topology (see \cite[Corollary 5.1.4]{BPO}).  
			Let $F_{\bcube}$ and $F_{U}$ as in Notation \ref{nota:big-to-small}. For $j\colon U\subseteq \P^1$, by \cite[Theorem 5.10]{BindaMerici} we have a short exact sequence in $\Sh(\P^1_{\Nis})$\[
			0\to F_{\bcube}\to j_*F_{U}\to \coker\to 0
			\]
			where $\coker$ is supported on $\P^1-U$, which has dimension 0, so $H^1_{\Nis}(\P^1,\coker) = 0$. Since $F\in \logCI$, $H^1_{\Nis}(\P^1,F_{\bcube}) = H^1_{\dNis}(\bcube,F) = 0$. Since we have that $R^ij_*=0$ for $i>0$, we conclude that\[
			H^1_{\dNis}(U,F) = H^1_{\Nis}(\P^1,j_*F_{U}) = 0
			\]
		\end{proof}
	\end{lemma}
	
	We give the following definition, extending  \cite[Definition 1.1]{Morel}:
	\begin{defn}\label{def:unramified}
		Let $F\in \PShlog(k,A)$. We say that $F$ is \emph{unramified} if  the following statements hold:
		\begin{enumerate}
			\item For any $X\in \SmlSm(k)$ with irreducible components $\{X_\alpha\}$, the obvious map $F(X)\to \prod_\alpha F(X_{\alpha})$ is a bijection.
			\item For any $X\in \SmlSm(k)$ and any dense strict open subscheme $U\subseteq X$ the restriction map $F(X) \to F(U)$ is injective.
			\item For any $X \in \SmlSm(k)$, irreducible with generic point $\eta_X$, the injective map $F(X) \hookrightarrow \bigcap_{x\in X^{(1)}} F(X_x)$ is a bijection of subobjects of $F(\eta_X)$.
		\end{enumerate}
	\end{defn}
	
	In this section, we will show the following result:
	
	\begin{thm}\label{thm:unramified}
		Let $F\in \logCI(k,A)$. Then $F$ is unramified.
	\end{thm}
	
	\begin{remark}\label{rmk:unramified}
		As observed in \cite[Remark 1.4]{Morel}, if $F$ satisfies $(1)$ and $(2)$, then $F$ satisfies $(3)$ if and only if it satisfies
		\begin{enumerate}
			\item[$(3')$] For all $Z\subseteq \ul{X}$ of codimension $\geq 2$, the map $F(X)\to F(\ul{X}-Z,\partial X_{|\ul{X}-Z})$ is a bijection.
		\end{enumerate}
		Notice that every $F\in \logCI(k,A)$ satisfies $(1)$, being a sheaf, and $(2)$ by \cite[Theorem 5.10]{BindaMerici}. Moreover, by \cite[Corollary 2.4]{shujilog}, for all $G\in \RSC_{\Nis}$, $\Log(G)$ satisfies $(3')$, hence it is unramified. 
	\end{remark}
	
	The key part of the proof of Theorem \ref{thm:unramified} is the following:
	\begin{lemma}\label{lem:key-unramified}
		Let $X\in \widetilde{\SmlSm(k)}$ such that $\ul{X}$ is a local scheme of dimension $n\geq 2$ with closed point $x$. Then for all $F\in \logCI(k,A)$ we have an isomorphism\[
		F(X)\cong F(X-x)\] and an injective map:\[H^1_{\dNis}(X,F)\hookrightarrow H^1_{\dNis}(X-x,F).
		\]
		\begin{proof}
  For $f\colon Y\to Y'$ in $\widetilde{\SmlSm(k)}$, let $\cofib(R\Gamma(f,F)):= \cofib(R\Gamma_{\dNis}(Y',F)\to R\Gamma_{\dNis}(Y,F))$. 
			Let $(X^h,x^h)$ be the henselization of $(X,x)$ and let $j\colon (X-x)\hookrightarrow X$ be the strict open immersion. Then we have a strict Nisnevich square:\[
			\begin{tikzcd}
				(X^h-x^h)\ar[r,"j^h"]\ar[d]&X^h\ar[d]\\
				(X-x)\ar[r,"j"] &X,
			\end{tikzcd}
			\] 
			in particular $\cofib(R\Gamma(j,F))\simeq \cofib(R\Gamma(j^h,F))$. By \cite[Theorem 3.5.7]{BPO-SH} (see also \cite[Theorem 7.7.4]{BPO}), we have that for all $Y\in \widetilde{\SmlSm}$ and $F\in \logCI(k,A)$ (see \cite[Theorem 5.1.8]{BPO})\[
            H^n_{\dNis}(Y,F) = \colim_{Y'\in Y^{\rm Sm}_{\rm div}} H^n_{\sNis}(Y',F),
            \] 
            and since $F$ is strictly $(\P^n,\P^{n-1})$-invatiant (as it is $(\bcube,\dNis)$-local) by \cite[Proposition 7.3.1]{BPO}), the latter is just $H^n_{\sNis}(Y,F)$ by \cite[Theorem 7.8.3]{BPO} (see also \cite[(7.8.1)]{BPO}). 
            Therefore, since $X^h$ is henselian, we have that $H^1_{\dNis}(X^h,F) \cong H^1_{\sNis}(X^h,F) = 0$, so we have a commutative diagram\[
			\begin{tikzcd}
				0\ar[r]&F(X)\ar[r,"F(j)"]\ar[d]&F(X-x)\ar[r]\ar[d]&\pi_0\cofib(R\Gamma(j,F))\ar[d,"\simeq"]\ar[r]&H^1_{\dNis}(X,F)\ar[d]\\
				0\ar[r]&F(X^h)\ar[r,"F(j^h)"]&F(X^h-x)\ar[r]&\pi_0\cofib(R\Gamma(j^h,F))\ar[r]&0.\\
			\end{tikzcd}
			\]
			In particular, it is enough to show that $\pi_0\cofib(R\Gamma(j^h,F))=\coker(F(j^h))=0$, hence we can suppose that $X$ is itself henselian. 
   By the usual standard argument (see \cite[Remark 4.6.1 (b)]{BeilinsonVologodsky}) we can assume that $k$ is infinite: indeed we can assume that there is a prime $\ell$ invertible in $A$ and take $k'/k$ a maximal pro-$\ell$-extension: then since $F$ has transfers we have a commutative diagram\[
   \begin{tikzcd}
       F(X_{k'})\ar[r]\ar[d]&F(X)\ar[d]\\
       F((X-x)_{k'})\ar[r]&F(X-x)
   \end{tikzcd}
   \]
   where the horizontal arrows are given by colimits over all finite subextensions $k'/L/k$ of the transpose of the finite maps $X_{L}\to X$ and $(X-x)_L\to (X-x)$.
   Let $r$ be the number of components of $\partial X$. As observed in \cite[Lemma 4.3]{BindaMerici} and \cite[Lemma 6.1]{SaitoPurity}, there is a regular sequence $(t_1\ldots t_n)$ in $\cO_{\ul{X}}(\ul{X})$  and an isomorphism $\ul{X}\cong \Spec(k(x)\{t_1,\ldots t_n\})$ such that $|\partial X|$ has local equation $t_1\ldots t_r$. If $r<n$, let $X'=(\ul{X},\partial X')$ where $|\partial X'|$ has local equation $t_1\ldots t_{r+1}$, and let $j'\colon X'\to X'-x$ be the open immersion. Let $D_{r+1}$ be the divisor with local equation $t_{r+1}=0$ on $\ul{X}$. Then as observed in \cite[(4.5.8)]{BindaMerici}, the localization sequence of \cite[Theorem 7.5.4]{BPO} induces a commutative diagram where the rows are exact and the columns are injective maps (in fact, the key passage of \cite[Claim 4.5]{BindaMerici} is the proof that the vertical map on the right is injective):\[
			\begin{tikzcd}
				0\ar[r]&F(X)\ar[r]\ar[d,hook,"F(j)"]&F(X')\ar[r]\ar[d,hook,"F(j')"]&\pi_{-1}\Map(MTh(N_{D_{r+1}/X}), F)\ar[r]\ar[d,hook]&0\\
				0\ar[r]&F(X-x)\ar[r]&F(X'-x)\ar[r]&\pi_{-1}\Map(MTh(N_{D_{r+1}-x/X-x}), F),
			\end{tikzcd}
			\] 
        where $\Map$ indicate the mapping spectrum, so we conclude then that $\coker F(j)\subseteq \coker F(j')$, so we can suppose that $r=n$. In this case, there is a strict pro-\'etale map \[
			X\to \bcube^{\times n}_{k(x)}
			\]
			which induces the following strict Nisnevich square in $\widetilde{\SmlSm(k)}$:
   \begin{equation}\label{eq:cd-reduce-henselian}
        \begin{tikzcd}
				(X-x)\ar[r,"j"]\ar[d]&X\ar[d]\\
				(\bcube^{\times n}_{k(x)}-\infty^{n})\ar[r,"j'"]&\bcube^{\times n}_{k(x)},
			\end{tikzcd}
	   \end{equation}
			where $\infty^{n}$ is the point $(\infty,\ldots,\infty)_{k(x)}$ of $\bcube^{\times n}_{k(x)}$, so we have $\cofib(R\Gamma(j,F))\simeq \cofib(R\Gamma(j',F))$. Since $F\in \logCI(k,A)$, we have that $H^1_{\dNis}(\bcube^{\times n}_{k(x)},F) = 0$, so we have following commutative diagram where the rows are short exact sequences:\[
			\begin{tikzcd}
				0\ar[r]&F(\bcube^{\times n}_{k(x)})\ar[r,"F(j')"]\ar[d]&F(\bcube^{\times n}_{k(x)}-\infty^n)\ar[r]\ar[d]&\pi_0\cofib(R\Gamma(j',F))\ar[r]\ar[d,"\cong"]&0\\
				0\ar[r]&F(X)\ar[r,"F(j)"]&F(X-x)\ar[r]&\pi_0\cofib(R\Gamma(j,F))\ar[r]&0,
			\end{tikzcd}
			\]
			hence it is enough to show that $F(j')$ is an isomorphism. Let $L\in \widetilde{\SmlSm(k)}$ be a field. For $Y\in \SmlSm(L)$ and $y\in Y^{\circ}$ an $L$-rational point, we let $\mathrm{str}\colon Y\to (\Spec(L),\triv)$ be the structural morphism and $i_y\colon (\Spec(L),\triv)\to Y$ be the closed immersion of $y$, and  $F(Y,i_y)$ be the complement of the induced split:
			\begin{equation}\label{eq:deg-zero}
				\begin{tikzcd}F(\Spec(L),\triv)\ar[r,"F({\rm str})"']&F(Y)\ar[l,bend right = 20, "F(i)"'],\end{tikzcd}
			\end{equation}
			In our case, we will fix the inclusion $i_{0^n}$ of $0^n:=(0,\ldots 0)_{k(x)}$ in $(\bcube^{\times n}_{k(x)}-\infty^n)$ and its open neighborhoods. To conclude, it is enough to show that $F((\bcube^{\times n}_{k(x)}-\infty^n),i_{0^n}) = 0$: this would imply that the map $F(j')$ is surjective. Consider the open subschemes $\A^1\times(\P^1)^{\times n-1}$ and $(\P^1)\times \A^1\times(\P^1)^{\times n-2}$ of $(\P^1)^{\times n}-\infty^{n}$: this induces a commutative square (not a cover!)
			\begin{equation}\label{eq:cover-P1-minus-infty}
				\begin{tikzcd}
					\A^2\times \bcube^{\times n-2}_{k(x)}\ar[r,"\alpha_1"]\ar[d,"\alpha_2"]&\A^1\times \bcube^{\times n-1}_{k(x)}\ar[d]\\
					\bcube\times \A^1\times \bcube^{\times n-2}_{k(x)}\ar[r]&\bcube^{\times n}_{k(x)}-\infty^n
				\end{tikzcd}
			\end{equation}
			Let $p_i\colon \A^2\to \A^1$ denote the $i$-th projection. By $\bcube$-invariance, we have for $i=0,1$:\[
			\begin{tikzcd}
				h_0^{\bcube}(\A^2\times \bcube^{\times n-2}_{k(x)})\ar[r,"h_0^{\bcube}\alpha_{i+1}"]\ar[d,"\cong"]&h_0^{\bcube}(\bcube^{\times i}\times \A^1\times \bcube^{\times n-1-i}_{k(x)})\ar[d,"\cong"]\\
				h_0^{\bcube}(\A^2\times k(x))\ar[r,"h_0^{\bcube}p_{i+1}"]&h_0^{\bcube}(\A^1\times k(x)).
			\end{tikzcd}
			\]
			Let $i_0\colon \Spec(k(x))\to \A^1_{k(x)}$ and $i_{(0,0)}\colon \Spec(k(x))\to \A^2_{k(x)}$ be the inclusions of the points $0$ and $(0,0)$ respectively: the maps $i_{0^n}$ clearly factor through $i_0$ and $i_{(0,0)}$ followed by the inclusions of $\A^1\times 0^{n-1}$, $0\times \A^1\times 0^{n-2}$ and $\A^2\times 0^{n-2}$, hence \eqref{eq:cover-P1-minus-infty} induces the following commutative square:
			\begin{equation}\label{eq:projections}
				\begin{tikzcd}
					F((\bcube^{\times r}_{k(x)}-\infty^n),i_{0^n})\ar[r,hook,"\ell_2"]\ar[d,hook,"\ell_1"]& F(\A^1_{k(x)},i_{0})\ar[d,"F(p_2^{0})"]\\
					F(\A^1_{k(x)},i_{0})\ar[r,"F(p_1^{0})"]& F(\A^2_{k(x)},i_{(0,0)})
				\end{tikzcd}
			\end{equation}
			where the maps $\ell_i$ are injective by purity of $F$ (\cite[Theorem 5.10]{BindaMerici}). We will conclude by showing that the map $\ell_1$ is the zero map. We have the following commutative diagrams:\[
			\begin{tikzcd}
				\Spec(k(x))\ar[d,equal]\ar[r,"i_0"]&\A^1_{k(x)}\ar[d,"id\times i_0"']\ar[l,bend left = 30,"{\rm str}"']\ar[dd,bend left = 50,equal]
				&&\Spec(k(x))\ar[d,equal]\ar[r,"i_0"]&\A^1_{k(x)}\ar[d,"id\times i_0"']\ar[l,bend left = 30,"{\rm str}"']\ar[dd,bend left = 50,"i_0\circ\textrm{ str}"]\\
				\Spec(k(x))\ar[r,"i_{(0,0)}"]\ar[d,equal]
				&\A^2_{k(x)}\ar[d,"p_1"']\ar[l,bend left = 30,"{\rm str}"']&&\Spec(k(x))\ar[r,"i_{(0,0)}"]\ar[d,equal]& \A^2_{k(x)}\ar[l,bend left = 30,"{\rm str}"']\ar[d,"p_2"']\\
				\Spec(k(x))\ar[r,"i_0"] &\A^1_{k(x)}\ar[l,bend left = 30,"{\rm str}"']&&\Spec(k(x))\ar[r,"i_0"]&\A^1_{k(x)}\ar[l,bend left = 30,"{\rm str}"']
			\end{tikzcd}
			\]
			which imply that we have a map $F((id\times i_0)^{0})\colon F(\A^2_{k(x)},i_{(0,0)})\to F(\A^1_{k(x)},i_{0})$ such that \[F((id\times i_0)^{0})\circ F(p_1^{0}) = id_{F(\A^1_{k(x)},i_{0})}\quad\textrm{and}\quad F((id\times i_0)^{0})\circ F(p_2^{0}) = 0,\]
			so by the commutativity of \eqref{eq:projections} we conclude that\[
			\ell_1 = F((id\times i_0)^{0})\circ F(p_1^{0})\circ \ell_1 = F((id\times i_0)^{0})\circ F(p_2^{0}) \circ \ell_2 = 0
			\]
		\end{proof}
	\end{lemma}
	
	\begin{proof}[Proof of Theorem \ref{thm:unramified}]
		We need to show that $F$ satisfies condition $(3')$ in Remark \ref{rmk:unramified}: let $X=(\underline{X},\partial X)\in \SmlSm(k)$ and $Z\subseteq \underline{X}$ of codimension $\geq 2$ and let $\eta_Z$ be its generic point. Recall that we have a strict Zariski square of log schemes\[
		\begin{tikzcd}
			X_{\eta_Z}-{\eta_Z}\ar[r]\ar[d]&X_{\eta_Z}\ar[d]\\
			X-Z\ar[r]&X
		\end{tikzcd}
		\]
		which induces by Lemma \ref{lem:key-unramified} a commutative diagram\[
		\begin{tikzcd}
			F(X) \ar[r]\ar[d]&F(X_{\eta_Z})\ar[d,"\simeq"]\\
			F(X-Z)\ar[r]&F(X_{\eta_Z}-{\eta_Z})
		\end{tikzcd}
		\]
		In particular, for $\alpha \in F(X-Z)$ there is $\beta \in F(X_{\eta_Z})$ such that $\beta\mapsto \alpha$ in $F(X_{\eta_Z}-\eta_Z)$, in particular there exists $U\subseteq X$ open with $U\cap Z$ dense in $Z$ and $\beta'\in F(U)$ such that $\beta\mapsto \alpha$ in $F(U-Z)$. Now let $Z_1$ be the complement of $U\cup (X-Z)$ in $X$: as $U\cap Z$ is dense in $Z$, we get that $\codim(Z_1)>\codim(Z)$. As $F$ is a sZar sheaf, this implies that there exists $\alpha_1\in F(X-Z_1)$ such that $\alpha_1\mapsto \alpha$ in $F(X-Z)$. 
		
		By repeating the same argument, we find a chain of strict closed subsets\[
		Z_r\subset \ldots \subset Z\subseteq \ul{X}
		\] 
		and elements $\alpha_i\in F(X-Z_i)$ such that $\alpha_i\mapsto \alpha$ in $F(X-Z)$. As $Z$ has finite Krull dimension, $Z_i=0$ for $i\gg 0$, which implies that $F(X)\to F(X-Z)$ is surjective. This together with \cite[Theorem 5.10]{BindaMerici} concludes the proof.
	\end{proof}
	
	We draw some conclusions from Theorem \ref{thm:unramified}. First we show the following interesting Gersten-like property:
	
	\begin{thm}\label{thm:Gersten}
		Let $X\in \SmlSm(k)$ and let $F\in \logCI(k,A)$. Let $F_{-1}:= \uExt^1_{\Sh_{\dNis}}(\Z(\P^1),F)$. Then there is a left exact sequence\[
		0\to F(X)\to F(\eta_X)\to \Bigl(\prod_{x\in (X^{\circ})^{(1)}} F(\A^1_{k(x)})/F(k(x))\oplus F_{-1}(k(x))\Bigr)\times \Bigl(\prod_{x\in |\partial X|^{(0)}} F(\A^1_{k(x)})/F(k(x))\Bigr).
		\]
		\begin{proof}
			Since $F$ is unramified by Theorem \ref{thm:unramified}, it is enough to show it on $X_x$ for $x\in \ul{X}^{(1)}$.
			Let $X^h_x\to X_x$ be the henselization and let $\eta_x^h$ be the generic point of $\ul{X_x^h}$: this gives a strict Nisnevich square
			\[
			\begin{tikzcd}
				(\eta_x^h,\triv)\ar[r,"j^h"]\ar[d] &X_x^h\ar[d]\\
				(\eta_X,\triv)\ar[r,"j"] &X_x
			\end{tikzcd}
			\]
			which gives the commutative diagram (see \eqref{eq:cd-reduce-henselian}) :
			\begin{equation}\label{eq:zar-to-Nis}
				\begin{tikzcd}
					0\ar[r]&F(X_x)\ar[r]\ar[d]&\omega_\sharp F(\eta_X)\ar[r]\ar[d]&\pi_0\cofib(R\Gamma(j,F))\ar[d,equal]\\
					0\ar[r]&F(X_x^h)\ar[r]&\omega_\sharp F(\eta_x^h)\ar[r]&\coker(F(j^h))\ar[r]&0
				\end{tikzcd}
			\end{equation}
			In particular, it is enough to show that for all $X_x^h\in \SmlSm(k)$ henselian DVR with log structure supported on the closed point and with generic point $\eta_x^h$, we have:
			\begin{equation}\label{eq:gersten-non-log}
				F(\eta_x^h)/F(\ul{X_x^h},\triv)\cong F(\A^1_{k(x)})/F(k(x))\oplus F_{-1}(k(x))
			\end{equation}
			and
			\begin{equation}\label{eq:gersten-log}
				F(\eta_x^h)/F(X_x^h)\cong F(\A^1_{k(x)})/F(k(x)).
			\end{equation}
			As in the proof of Lemma \ref{lem:key-unramified} (see also \cite[Theorem 4.1, (ii)]{BindaMerici}), there is a formally \'etale map $\ul{X_x^h}\to \P^1_{k(x)}$ inducing strict Nisnevich squares:
			\begin{equation}\label{eq:Gabber}
				\begin{tikzcd}
					(\eta_x^h,\triv) \ar[r]\ar[d,"{(e^\circ)}"] &(\ul{X_x^h},\triv) \ar[d,"\ul{e}"]&&	(\eta_x^h,\triv) \ar[r]\ar[d,"{(e^\circ)}"] &X_x^h\ar[d,"e"]\\
					(\A^1_{k(x)},\triv)\ar[r] &(\P^1_{k(x)},\triv)&&(\A^1_{k(x)},\triv)\ar[r] &\bcube_{k(x)}.
				\end{tikzcd}
			\end{equation}
			Since $F\in \logCI(k,A)$, we have that $F(\bcube_{k(x)})\cong F(k(x))$ and $H^1_{\dNis}(\bcube_{k(x)},F) = 0$. The square
	\begin{equation}\label{eq:MV-P1}
		\begin{tikzcd}
			(\P^1,0+\infty)\ar[r]\ar[d]&(\P^1,0)\ar[d]\\
			(\P^1,\infty)\ar[r]&(\P^1,\triv),
		\end{tikzcd}
	\end{equation}
	induces a cartesian square in $\logDM(k,\Z)$ (see \cite[Proposition 3.1.6]{BPO-SH}), therefore $F(\P^1_{k(x)})\simeq F(k(x))$.
Finally, by Lemma \ref{lem:cohomology-open-P1} we have $H^1_{\dNis}(\A^1_{k(x)},F) =0$. This implies that we have short exact sequences
			\begin{equation}\label{eq:sequence-coh-P1}
				\begin{tikzcd}
					0\ar[r]&F(\A^1_{k(x)})/F(k(x))\ar[r]\ar[d,equal] &F(\eta_x^h)/F(\ul{X_x^h})\ar[r]\ar[d] &H^1_{\Nis}(\P^1_{k(x)},\omega_\sharp F)\ar[d]\ar[r]&0\\
					0\ar[r]&F(\A^1_{k(x)})/F(k(x))\ar[r,"\cong"] &F(\eta_x^h)/F(X_x^h)\ar[r]&0.
				\end{tikzcd}
			\end{equation}
			This implies that the above sequence splits and concludes the proof.
		\end{proof}
	\end{thm}
	We recall that there is a functor $\omega_\sharp \colon \PShlog(k,A)\to \PSh(k,A)$ defined as $\omega_\sharp F(X) = F(X,\triv)$.
	\begin{defn}\label{def:residues}
		A \emph{presheaf with residues} is the datum of $F\in \PSh(k,A)$ and for all $X\in \widetilde{\Sm(k)}$ the spectrum of an henselian DVR with generic point $\eta_X$ and residue field $k(x)$, a map $F(\eta_X)\to F(\A^1_{k(x)})/F(k(x))$. 
	\end{defn}
	
	\begin{remark}\label{rmk:residues}
		Let $F\in \logCI$, then $\omega_{\sharp}F$ is naturally a presheaf with residues by \eqref{eq:sequence-coh-P1}. Moreover, for all $X_x^h= (\Spec(k(x)\{t\}),t)\in \widetilde{\SmlSm(k)}$ and $\ul{X_x^h}$ the underlying scheme with trivial log structure, we get the following variant of \eqref{eq:Gabber}:\[
		\begin{tikzcd}
			(\eta_x^h,\triv) \ar[r]\ar[d,"{(e^\circ)}"] &(\ul{X_x^h},\triv) \ar[d,"\ul{e}"]&&	(\eta_x^h,\triv) \ar[r]\ar[d,"{(e^\circ)}"] &X_x^h\ar[d,"e"]\\
			(\A^1_{k(x)}-\{0\},\triv)\ar[r] &(\A^1_{k(x)},\triv)&&(\A^1_{k(x)}-\{0\},\triv)\ar[r] &(\A^1_{k(x)},0).
		\end{tikzcd}
		\]
		together with the Nisnevich squares\[
		\begin{tikzcd}
			(\A^1_{k(x)}-\{0\},\triv)\ar[r]\ar[d] &(\A^1_{k(x)},\triv)\ar[d]&&(\A^1_{k(x)}-\{0\},\triv)\ar[r]\ar[d] &(\A^1_{k(x)},0)\ar[d]\\
			(\P^1_{k(x)}-\{0\},\triv)\ar[r] &(\P^1_{k(x)},\triv)&&
			(\P^1_{k(x)}-\{0\},\triv)\ar[r] &(\P^1_{k(x)},0).
		\end{tikzcd}
		\]
		we have the a commutative diagram (we drop the ``$\triv$''):\[
		\begin{tikzcd}
			&\pi_0\cofib(R\Gamma_{\Nis}(\P^1_{k(x)},\omega_\sharp F)\to F(\A^1_{k(x)}))\ar[r]\ar[dd,"\simeq"] &\frac{ F(\A^1_{k(x)})}{F(\bcube_{k(x)})}\ar[dd,"\simeq"]\\
			\frac{F(\eta_x^h)}{F(\ul{X_x^h})}\ar[ur,"\simeq"]\ar[dr,"\simeq"]&&&	\frac{F(\eta_x^h)}{F(X_x^h)}\ar[ul,"\simeq"']\ar[dl,"\simeq"'] \\
			&\frac{F(\A^1_{k(x)}-\{0\})}{F(\A^1_{k(x)})}\ar[r]&\frac{ F(\A^1_{k(x)}-\{0\})}{F(\A^1_{k(x)},0)}
		\end{tikzcd}
		\]
		which implies by $\bcube$-invariance that one can construct residues via the map on the bottom. This has the advantage that there is no higher cohomology involved in the construction.
	\end{remark}
	The main application of the construction above is the following:
	\begin{prop}\label{prop:residues-lift}
		For $F,G\in \logCI$, any map $\phi\colon \omega_\sharp F\to \omega_\sharp G$ of presheaves with residues lifts uniquely to a map of log sheaves $\widetilde{\phi}\colon F\to G$. If $F,G\in \logCI^{\ltr}$ and $\phi$ is a map of presheaves with transfers, then $\widetilde{\phi}$ is a map of log presheaves with transfers.
		\begin{proof}
			The uniqueness is \cite[Proposition 0.1]{BindaMericierratum}. Let $X\in \SmlSm$ and let $j\colon (X^\circ,\triv)\to X$ be the inclusion of the trivial locus. First we observe that it is enough to show that for all $X\in \SmlSm(k)$ irreducible there is a map $\phi_{X}\colon F(X)\to G(X)$ that makes the following diagram commute:
			\begin{equation}\label{eq:conj-reduction-1}
				\begin{tikzcd}
					F(X)\ar[r,hook,"F(j)"]\ar[d,dotted,"\phi_{X}"] &\omega_\sharp F(X^{\circ})\ar[d,"\phi_{X^{\circ}}"]\\
					G(X)\ar[r,hook,"G(j)"] &\omega_\sharp G(X^{\circ}).
				\end{tikzcd}
			\end{equation}
			Let $\alpha\colon Y\to X$ a map of log schemes and assume that $\phi_X$ and $\phi_Y$ exist. Consider $\alpha^{\circ}\colon Y^\circ\to X^\circ$ the restriction to the trivial loci, then since $\omega_\sharp \phi$ is a map of presheaves we have a commutative diagram:
			\begin{equation}\label{eq:conj-reduction-2}
				\begin{tikzcd}
					\omega_\sharp F(X^{\circ})\ar[r,"\omega_\sharp F(\alpha^{\circ})"]\ar[d,"\phi_{X^{\circ}}"]&\omega_\sharp F(Y^\circ)\ar[d,"\phi_{Y^{\circ}}"]\\
					\omega_\sharp G(X^{\circ})\ar[r,"\omega_\sharp G(\alpha^{\circ})"]&\omega_\sharp G(Y^{\circ}),
				\end{tikzcd}
			\end{equation}
			Since the map $G(j)$ above is injective by \cite[Theorem 5.10]{BindaMerici}, it is enough to show that $G(j)\phi_YF(\alpha) = G(j)G(\alpha)\phi_X$, which follows from the commutativity of \eqref{eq:conj-reduction-1} for $X$ and $Y$ and of \eqref{eq:conj-reduction-2}.
			
			We now show the existence of the map $\phi_{X}$ in \eqref{eq:conj-reduction-1}. Let $\eta_X$ be the generic point of $\ul{X}$: since $F$ and $G$ are unramified by Theorem \ref{thm:unramified}, by the property (3) in Definition \ref{def:unramified}, it is enough to check that for all $x\in X$ of codimension $1$ we have a map $\phi_{X_x}$ that fits in the following commutative diagram:\[
			\begin{tikzcd}
				F(X_x)\ar[r,hook]\ar[d,dotted,"{\phi_{X_x}}"] &\omega_\sharp F(\eta_X)\ar[d,"{\phi_{\eta_X}}"]\\
				G(X_x)\ar[r,hook] &\omega_\sharp G(\eta_X)
			\end{tikzcd}
			\]
			If $x\not \in |\partial X|$, $X_x$ has trivial log structure so $\phi_{X_x}$ exists, so assume $x\in |\partial X|$. As in \eqref{eq:zar-to-Nis} it is enough to construct a map $F(X_x^h)\dashrightarrow G(X_x^h)$. Since $\phi$ preserves residues, we have a commutative diagram:\[
			\begin{tikzcd}
				\omega_\sharp F(\eta_x^h)\ar[rr,"\Res(F(\ul{X_x^h}))"]\ar[d,"\phi_{\eta_x^h}"]&&\omega_\sharp F(\A^1_{k(x)})/\omega_\sharp F(k(x))\ar[d,"\Res(\phi)"]\\
				\omega_\sharp G(\eta_x^h)\ar[rr,"\Res(G(\ul{X_x^h}))"]&&\omega_\sharp G(\A^1_{k(x)})/\omega_\sharp G(k(x)).
			\end{tikzcd}
			\]
			On the other hand, by Theorem \ref{thm:Gersten} we have that $\ker(\Res (F(\ul{X_x^h})))=F({X_x^h})$ and  same for $G$, so $\phi_{X_x^h}$ exists as the map induced on the kernels. If $\phi$ is a map of presheaves with transfers, then the commutativity of \eqref{eq:conj-reduction-1} and \eqref{eq:conj-reduction-2} for $\alpha$ a log finite correspondence implies that $\phi$ extends to a map with transfers.
		\end{proof}	
	\end{prop}
	
	We are now ready to prove the main theorem of this section. First let us fix some notation: let $k(x)$ be as before and let $U\subseteq \A^1_{k(x)}$ be an open neighborhood of $\{0\}$. Then for all $F\in \logCI$, by \cite[Theorem 7.5.4]{BPO} there is a fiber sequence
	\begin{equation}\label{eq:Thom}
		\begin{aligned}
			\Map_{\logDA}(MTh(N_{\{0\}}),F)\to \Map_{\logDA}(M(U,\triv),F)\to  \Map_{\logDA}(M(U,0),F)
		\end{aligned}
	\end{equation}
	compatible with open embeddings $V\hookrightarrow U$. Recall now that 
	\begin{equation}\label{eq:Thom-P1}
		\Map_{\logDA}(MTh(N_{\{0\}}),F) \cong R\Gamma_{\dNis}(((\P^1_{k(x)},\triv),i_{0}),F) = H^1_{\Nis}(\P^1_{k(x)},\omega_{\sharp} F)[-1].
	\end{equation}
	Let $i\colon \{0\}\hookrightarrow\A^1_{k(x)}$: combining \eqref{eq:Thom} and \eqref{eq:Thom-P1} gives a Gysin map in $\cD(\Sh((\A^1_{k(x)})_{\Zar}))$\[
	\mathrm{Gys}_{0}^F\colon i_*H^1({\P^1_{k(x)}},\omega_\sharp F)\to F_{\A^1_{k(x)}}[1],
	\] 
	where $H^1(\P^1_{k(x)},\omega_{\sharp}F)$ on the left hand side is seen as a (constant) sheaf on the small site of $\{0\}$. Via the isomorphism \begin{equation}\begin{aligned}
			&\pi_0\Map_{\cD(\Sh((\A^1_{k(x)})_{\Zar}))}(i_*H^1({\P^1_{k(x)}},\omega_\sharp F), F_{\A^1_{k(x)}}[1])\\
			&\cong \Ext^1_{\Sh((\A^1_{k(x)})_{\Zar})}(i_*H^1({\P^1_{k(x)}},\omega_\sharp F), F_{\A^1_{k(x)}}),\end{aligned}
	\end{equation} the map $\mathrm{Gys}_{0}^F$ above corresponds to the extension:
	\begin{equation}\label{eq:Gysin-es}
		0\to \omega_\sharp F_{\A^1_{k(x)}}\to F_{(\A^1_{k(x)},0)}\to i_*H^1(\P^1_{k(x)},\omega_\sharp F)\to 0
	\end{equation}
	\begin{thm}\label{thm:Gysin}
		Let $F,G\in \logCI$, and let $\phi\colon \omega_\sharp F\to \omega_\sharp G$ be a map. If the Gysin maps induce a commutative diagram in $\cD(\Sh((\A^1_{k(x)})_{\Zar}))$\[
		\begin{tikzcd}
			i_*H^1({\P^1_{k(x)}},\omega_\sharp F)\ar[r,"\mathrm{Gys}_{0}^F"]\ar[d,"i_*H^1(\phi_{\P^1_{k(x)}})"'] &F_{\A^1_{k(x)}}[1]\ar[d,"\phi_{\A^1_{k(x)}}"]\\
			i_*H^1({\P^1_{k(x)}},\omega_\sharp G)\ar[r,"\mathrm{Gys}_{0}^G"] &G_{\A^1_{k(x)}}[1],
		\end{tikzcd}
		\]
		for all finitely generated field extensions $k(x)/k$, then there is a unique map $\widetilde{\phi}\colon F\to G$ that lifts $\phi$. If $F,G\in \logCI^{\ltr}$ and $\phi$ is a map of presheaves with transfers, then $\widetilde{\phi}$ is a map of log presheaves with transfers.
		\begin{proof}
			The uniqueness of $\widetilde{\phi}$ is \cite[Proposition 0.1]{BindaMericierratum}. The commutative diagram in the statement gives a morphism $\tilde\phi_{\A^1}$ that fits in the exact sequences of \eqref{eq:Gysin-es} as follows:\[
			\begin{tikzcd}
				0\ar[r] &\omega_\sharp F_{\A^1_{k(x)}}\ar[r]\ar[d,"\phi_{\A^1}"] &F_{(\A^1_{k(x)},0)}\ar[d,"{\tilde\phi_{\A^1}}"]\ar[r] &i_*H^1(\P^1_{k(x)},\omega_\sharp F)\ar[r]\ar[d] &0\\
				0\ar[r] &\omega_\sharp G_{\A^1_{k(x)}}\ar[r] &G_{(\A^1_{k(x)},0)}\ar[r] &i_*H^1(\P^1_{k(x)},\omega_\sharp G)\ar[r] &0.
			\end{tikzcd}
			\]
			By construction, we have that $(\tilde\phi_{\A^1})_{|\A^1-\{0\}} = \phi_{\A^1-\{0\}}$, so we have a commutative diagram\[
			\begin{tikzcd}
		      F{(\A^1_{k(x)},0)}\ar[r]\ar[d,"\tilde\phi_{\A^1}"]&\omega_\sharp F(\A^1_{k(x)}-\{0\})\ar[d,"\phi_{\A^1-\{0\}}"]\\
				G(\A^1_{k(x)},0)\ar[r]&\omega_\sharp G(\A^1_{k(x)}-\{0\}),
			\end{tikzcd}
			\]
			As observed in Remark \ref{rmk:residues}, this implies that $\phi$ is a map of presheaves with residues, which by Proposition \ref{prop:residues-lift} gives the desired lift, which is a map of presheaves with log transfers if $\phi$ has log transfers by Proposition \ref{prop:residues-lift}.
		\end{proof}
	\end{thm}
	From Theorem \ref{thm:Gysin}, we will deduce the transfer structure on $W_m\Lambda^n$ in Theorem \ref{thm:de-rham-witt-rsc-log}.
	\section{Application to log de Rham--Witt cohomology}\label{sec:log-DRW}
	
	Let $S$ be a Noetherian $k$-scheme and let $W_m\Lambda^q\in \Sh_{\sZar}(S,\Z)$ of \cite{matsuue} (see \ref{ssec:logdRW}).
	The idea behind the following result has been suggested by F. Binda. We thank him for letting us include it here.
	Let $X\in \SmlSm(S)$. Consider the exact sequence on the weight filtration \eqref{eq:weight-es} for $X\times(\P^n,\P^{n-1})$. 
	We have that $\partial(X\times(\P^n,\P^{n-1}))^{(j)} = \partial X^{(j)}\times \P^n \coprod \partial X^{(j-1)}\times \P^{n-1}$, so if again we let $i_H\colon \P^{n-1}\to \P^n$ be the inclusion of a fixed hyperplane, we let $\alpha^{(j)}= \iota^{(j)}\times id\colon \partial X^{(j)}\times \P^n\to \ul{X}\times \P^n$ and $\beta^{(j)}= \iota^{(j-1)}\times i_H\colon \partial X^{(j-1)}\times \P^{n-1}\to \ul{X}\times \P^n$, we get an exact sequence:
	\begin{equation}\label{eq:weight-es-pn}
		\begin{aligned}	0&\to P_{j}W_m\Lambda^q_{X\times(\P^n,\P^{n-1})}\to P_{j+1}W_m\Lambda^q_{X\times (\P^n,\P^{n-1})}\\
			&\xrightarrow{{\rm Res}} \alpha^{(j+1)}_*W_m\Omega^{q-(j+1)}_{\partial X^{(j+1)}\times \P^n} \oplus \beta^{(j+1)}_*W_m\Omega^{q-(j+1)}_{\partial X^{(j)}\times \P^{n-1}}\to 0
		\end{aligned}
	\end{equation}
	
	\begin{lemma}\label{lem:cohomology-weight-filtration}
		For $X\in \SmlSm(S)$ and $j\geq 1$ there is an equivalence\[
		R\pi_*P_jW_m\Lambda^q_{X\times (\P^n,\P^{n-1})} \cong P_jW_m\Lambda^q_{X}[0]\oplus  \bigoplus_{i=1}^{n}\iota^{(j)}_*W_m\Omega^{q-j-i}_{\partial X^{(j)}}[-i]  \quad \mathrm{in}\quad  \cD (\Sh(\ul{X}_{\rm Zar}))
		\]
	\end{lemma}
	\begin{proof}
		To ease the notation, let $\cP^n:=(\P^n,\P^{n-1})$. 
		
		We prove the statement by induction on $j$. If $j=0$, then the statement is just the projective bundle formula. 
		Let $j\geq 1$. Then \eqref{eq:weight-es-pn} gives the fiber sequence\[
		R\pi_*P_{j-1}W_m\Omega^q_{X\times \cP^n}\to R\pi_*P_jW_m\Lambda^q_{X\times \cP^n} \to R(\pi_H^{\partial X^{(j)}})_*W_m\Omega^{q-j}_{\P^{n-1}_{\partial X^{(j-1)}}}\oplus R(\pi{\partial X^{(j)}})_*W_m\Omega^{q-j}_{\P^{n}_{\partial X^{(j)}}},
		\]
		where $\pi^{\partial X^{(j)}}\colon \P^n_{\partial X^{(j)}}\to \partial X^{(j)}$ and $\pi^{\partial X^{(j-1)}}_H\colon \P^{n-1}_{\partial X^{(j-1)}}\to \partial X^{(j-1)}$ are the projections. By induction hypothesis and projective bundle formula we have a commutative diagram\[
		\begin{tikzcd}
			\bigoplus_{i=0}^{n} i_*^{(j)} W_m\Omega^{q-j-i}_{\partial X^{(j)}}[-i-1]\ar[phantom,d,"\oplus"]\ar[phantom,d,shift left =15,""{name=D1}]&R(\pi^{\partial X^{(j)}})_*W_m\Omega^{q-j}_{\P^{n}_{\partial X^{(j)}}}[-1]\ar[phantom,d,"\oplus"]\ar[phantom,d,shift right = 15, ""{name=D2}]\\
			\bigoplus_{i=0}^{n-1} W_m\Omega^{q-j-i}_{\partial X^{(j-1)}}[-i-1] \ar[d,"t"] &R(\pi^{\partial X^{(j-1)}}_H)_*W_m\Omega^{q-j}_{\P^{n-1}_{\partial X^{(j-1)}}}[-1]\ar[d,"f"]\\
			P_{j-1}W_m\Omega^q_{\ul{X}}\oplus \bigoplus_{i=1}^{n}	\iota_{*}^{(j-1)}W_m\Omega^{q-(j-1)-i}_{\partial X^{(j-1)}}[-i] \ar[r,"\simeq"] &R\pi_*P_{j-1}W_m\Lambda^q_{\cP^n_{X}}
			\arrow[from=D1,to=D2,"\simeq"]
		\end{tikzcd}
		\]
		Again, we can compute its cofiber as the cofiber of\[
		\bigoplus_{i=0}^{n} \iota_*^{(j)} W_m\Omega^{q-j-i}_{\partial X^{(j)}}[-i-1]\to P_{j-1}W_m\Omega^q_{\ul{X}}
		\]
		which again by \eqref{eq:weight-es} is\[
		P_jW_m\Omega^q_{\ul{X}}\oplus \bigoplus_{i=1}^n W_m\Omega^{q-i-1}_{\partial X^{(j)}}[-i].
		\]
		This concludes the proof.
	\end{proof}
	We are ready to show the following:
	\begin{thm}\label{thm:witt-local}
		The sheaf $W_m\Lambda^{q}$ is $(\P^n,\P^{n-1})$-local in the $\sZar$, $\sNis$ and $\mathrm{k\et}$ topologies. In particular, it is strictly $\div$-invariant.
		\begin{proof}By \cite[Proposition 3.27]{Niziol}, we have that for all $Y\in \SmlSm(S)$\[
			R\Gamma_{\sNis}(Y,W_m\Lambda^q_{S})\simeq R\Gamma_{\sZar}(Y,W_m\Lambda^q_{S})\simeq  R\Gamma_{\mathrm{k\et}}(Y,W_m\Lambda^q_{S})
			\]
			hence it is enough to show it only for the $\sZar$ topology.
			Let $P_j$ again be the weight filtration defined above.
			We have by definition that $W_m\Lambda^q \cong P_qW_m\Lambda^q$, so since $W_m\Lambda^{q-q-i}_{\partial X^{(q)}} = 0$ for $i\geq 1$, we conclude by Lemma \ref{lem:cohomology-weight-filtration} that $R\pi_*W_m\Lambda^q_{X\times (\P^n,\P^{n-1})}\simeq W_m\Lambda^q_X$. Finally by \cite[Theorem 7.7.4]{BPO}, any $(\sNis,(\P^\bullet,\P^{\bullet-1}))$-local object is automatically $\div$-local.
		\end{proof}
	\end{thm}
	
	\begin{remark}\label{rmk:log-crys}
		Theorem \ref{thm:witt-local} implies that for every $m$, any $\sNis$ (resp. $\mathrm{k\et}$)-fibrant replacement of the truncated de Rham--Witt complex $W_m\Lambda_{S}^{\bullet}$ is $(\dNis,\bcube)$-fibrant. Hence, the functor \[
		\SmlSm(S) \to \cD(W(k))^{op}\qquad X/S\mapsto \holim_m R\Gamma(X,W_m\Lambda^{\bullet}_{X/S})
		\]
		factors through $\mathbf{log}\mathcal{DA}^{\rm eff}(S,\Z)$ and $\mathbf{log}\mathcal{DA}^{\rm eff}_{\loget}(S,\Z)$ via the functor:
		\begin{align*}
			R\Gamma_{\mathrm{crys},S}\colon &\mathbf{log}\mathcal{DA}^{\rm eff}_{\loget}(S,\Z)\to \cD(W(k))^{op}\\
			&M\mapsto \lim_m\Map(M,W_m\Lambda_{-/S}^{\bullet}).
		\end{align*}
	\end{remark}
	We now concentrate on the case where $S=\Spec(k)$. In this case, since $W_m\Omega^n\in \RSC_{\Nis}(k)$ by \cite[Appendix B]{KSY}, we have $\Log(W_m\Omega^n)\in \logCI^{\ltr}$. Let $\ul{X}\in \Sm(k)$ and $i\colon D\subseteq X$ a smooth divisor, and let $X=(\ul{X},D)\in \SmlSm(k)$. By \eqref{eq:BRS-Gysin-ext} there is an exact sequence of $\Sh(\ul{X}_{\Nis})$:
	\begin{equation}\label{eq:almost-es-log}
		0\to W_m\Omega^n_{\ul{X}} \to  \Log(W_m\Omega^n)_{X} \to \gamma^1 W_m\Omega^n_{D}\to 0
	\end{equation}
	Moreover, by \cite[Theorem 11.8]{BindaRuellingSaito} there is an isomorphism\[
	\phi^1\colon W_m\Omega^{n-1}\cong \gamma^1 W_m\Omega^n
	\]
	that behaves as follows: let $Y\in \Sm(k)$ and let $\lambda_Y\colon  \gamma^1 W_m\Omega^n(Y)\hookrightarrow  W_m\Omega^n(Y\times (\A^1-\{0\}))$ be the map induced by $\Z_{\tr}(\A^1-\{0\})\twoheadrightarrow \G_m$. Recall the map of $\Sh((\A^1-\{0\})_{\Zar})$\[
	\dlog\colon (W_m\cO_{\A^1-\{0\}})^{\times} =
	\cO_{\A^1_{W_m(k)}-\{0\}}^{\times}\xrightarrow{f\mapsto \frac{df}{f}} Z\Omega^1_{\A^1_{W_m}-\{0\}/W_m}\twoheadrightarrow W_m\Omega^1_{\A^1-\{0\}}.\] 
	Then for $\omega\in W_m\Omega^{n-1}(Y)$ we have:
	\begin{equation}\label{eq:dlog}
		\lambda_Y\phi^1_Y (\omega) = \omega \cdot \dlog(\ul{t}) \quad\textrm{in }W_m\Omega^n(Y\times (\A^1-\{0\})),
	\end{equation}
	where $\ul{t}\in W_m\cO^{\times}(\A^1-\{0\})$ is the Teichm\"uller lift of $t$. Then \eqref{eq:almost-es-log} translates in the following exact sequence in $\Sh(\ul{X}_{\Nis})$:
	\begin{equation}\label{eq:es-log}
		0\to W_m\Omega^n_{\ul{X}}\to \Log(W_m\Omega^n)_{X}\to i_*W_m\Omega^{n-1}_{D}\to 0,
	\end{equation}
	\begin{thm}\label{thm:de-rham-witt-rsc-log}
		There is an isomorphism of sheaves $W_m\Lambda^n\simeq\Log(W_m\Omega^n)$. In particular $W_m\Lambda^n\in \logCI^{\ltr}(k,W_m(k))$ and the Frobenius, Verschiebung, differential and restrictions are morphisms of log presheaves with transfers.
		\begin{proof}
			Let $\Log(W_m\Omega^n)\in \logCI^{\ltr}(k,W_m(k))$. We have that 
			\begin{equation}\label{eq:iso-drw}
				\omega_\sharp\Log(W_m\Omega^n) \cong W_m\Omega^n \cong \omega_\sharp W_m\Lambda^n
			\end{equation}
			We will show that \eqref{eq:iso-drw} lifts to an isomorphism $W_m\Lambda^n\cong \Log(W_m\Omega^n)$ in $\logCI(k,W_m(k))$. This will give the transfers structure on $W_m\Lambda^n$ in the following way: for $\alpha\in \lCor(Y,X)$, we have a commutative diagram:\[
			\begin{tikzcd}
				W_m\Lambda^n(X)\ar[r,"\simeq"] &\Log(W_m\Omega^n)(X)\ar[r,hook]\ar[d,"\alpha^*"]&W_m\Omega^n(X^\circ)\ar[d,"(\alpha^{\circ})^*"]\\
				W_m\Lambda^n(Y)\ar[r,"\simeq"]&\Log(W_m\Omega^n)(Y)\ar[r,hook] &W_m\Omega^n(Y^\circ)
			\end{tikzcd}
			\]
			The morphisms $F,V,d$ and $R$ preserve transfers by \cite[Lemma 7.7]{RulSaito}.
			
			We need to study the maps $i_*W_m\Omega^{n-1}_{k(x)}\to W_m\Omega^n_{\A^1_{k(x)}}[1]$ given by the two Gysin maps in $D(\Sh((\A^1_{k(x)})_{\Zar})$: if the two are the same, then we conclude by Theorem \ref{thm:Gysin}. 
			By \cite[Proposition 4.7]{mokrane}, the Gysin map of $W_m\Lambda^n$ agrees with the Gysin map of \cite{Gros1985}, so it is enough to show that the Gysin maps of \cite{BindaRuellingSaito} and  \cite{Gros1985} are the same. 
			This was stated without proof in \cite[11.10]{BindaRuellingSaito}: a proof was communicated to the author by Kay R\"ulling and we write it here. By construction both maps factor through the local Gysin maps:
			\begin{align*}
				&g_{0/\A^1}^{\rm Gros}\colon i_*W_m\Omega^{n-1}_{k(x)}\to R^1\ul{\Gamma}_{\{0\}}(W_m\Omega^n_{\A^1})\\
				&i_*W_m\Omega^{n-1}_{k(x)}\overset{i_*\phi^1_{k(x)}}{\cong} i_*(\gamma^1W_m\Omega^n)_{k(x)} \xrightarrow{g_{0/\A^1}^{\rm BRS}} R^1\ul{\Gamma}_{\{0\}}(W_m\Omega^n_{\A^1})
			\end{align*}
			of \eqref{eq:Gros-Gysin} and \eqref{eq:BRS-Gysin} composed with $\phi^1$ of \cite[Theorem 11.8]{BindaRuellingSaito}, hence it is enough that they agree before the ``forget support" map. Let us fix the isomorphism\[
			R^1\ul{\Gamma}_{\{0\}}(W_m\Omega^n_{\A^1})\cong j_*W_m\Omega^{n}_{\A^1_{k(x)}-\{0\}}/W_m\Omega^n_{\A^1_{k(x)}}.
			\] 
			Let $t$ be the local equation of $0$ in $\A^1$, and let $U$ be an open neighborhood of $0$. Let  $0\xrightarrow{i} \A^1\xrightarrow{q} 0$ be the splitting of the inclusion. Let $\omega\in W_m\Omega^{n-1}(k(x))$. By \cite[II, Prop 3.3.9]{Gros1985} we have that\[
			(g_{0/\A^1}^{\rm Gros})_U(\omega) = [q^*\omega\cdot \dlog(\ul{t})]\quad\textrm{in } W_m\Omega^{n}(U-\{0\})/W_m\Omega^n(U) \cong W_m\Omega^{n}(\A^1_{k(x)}-\{0\})/W_m\Omega^n(\A^1_{k(x)})
			\] 
			where $\ul{t}\in W_m\cO^{\times}(\A^1-\{0\})$ is the Teichm\"uller lift of $t$. 
			
			On the other hand, we have that the map\[
			i_*W_m\Omega^{n-1}(U)\to i_*q_*\gamma^1W_m\Omega^{n}(U)
			\]
			coincides with the map $W_m\Omega^{n-1}(k(x))\to \gamma^1W_m\Omega^{n}(\A^1_{k(x)})$ induced by $q$ and $\phi^1$. Moreover, the map\[
			i_*q_*\gamma^1W_m\Omega^{n}(U)\to R^1\ul{\Gamma}_{\{0\}}(W_m\Omega^n_{\A^1})(U)
			\]
			coincides with the map\[
			\gamma^1W_m\Omega^{n}(\A^1_{k(x)})\to W_m\Omega^n(\A^1_{k(x)}-\{0\})/W_m\Omega^n(\A^1_{k(x)}) \cong W_m\Omega^n(U-\{0\})/W_m\Omega^n(U)
			\]
			given by \eqref{eq:BRS-cycle-class} evaluated at $0\subseteq \A^1_{k(x)}$. In particular, for $\omega\in W_m\Omega^{n-1}(k(x))$, by unwinding the definition of \eqref{eq:BRS-local-Gysin} we have that $g_{0/\A^1}^{\textrm{BRS}}i_*\phi_{k(x)}(\omega)$ agrees with the class of\[
			[\Delta^*(\Gamma_t \times id)^*\lambda_{\A^1}^*\phi_{\A^1}^*q^*(\omega)] \quad\textrm{in } W_m\Omega^n(\A^1_{k(x)}-\{0\})/W_m\Omega^n(\A^1_{k(x)})
			\]
			By \eqref{eq:dlog} we have that \[
			\lambda_{\A^1}^*\phi_{\A^1}^*q^*(\omega) = q^*\omega\mapsto q^*\omega\wedge \dlog(\ul{t}).
			\]
			Moreover, we have that\[
			(\Gamma_t \times id)^*(q^*\omega\wedge  \dlog(\ul{t})) = pr_2^*(q^*\omega)\wedge \dlog(\ul{t}\otimes 1) \quad\textrm{in } W_m\Omega^{n}(\A^1_{k(x)}\times\A^1-\{0\}),
			\]
			where $pr_2$ is the second projection. By composing with $\Delta^*$ we conclude that \[
			g_{0/\A^1}^{\textrm{BRS}}i_*\phi_{k(x)}(\omega) = [\Delta^*(pr_2^*(q^*\omega)\wedge \dlog(\ul{t}\otimes 1))] = [q^*\omega\cdot \dlog(\ul{t})] = g_{0/\A^1}^{\textrm{Gros}}(\omega).
			\]
			This concludes the proof.
		\end{proof}
	\end{thm}
	\begin{cor}\label{cor:dRW-P1}
		We have that\[
		\uExt^i(\Z_{\ltr}(\P^1,\triv),W_m\Lambda^n) \cong \begin{cases}
			W_m\Lambda^n &\textrm{if }i=0\\
			W_m\Lambda^{n-1} &\textrm{if }i=1\\
			0 &\textrm{otherwise}
		\end{cases}
		\]
	\end{cor}
	\begin{proof}
		The case $i=0$ follows from the $\P^1$-invariance of objects in $\logCI$.
		Let $T\in \widetilde{\SmlSm(k)}$ be Hensel local and let $\eta_T$ be its generic point. Then by \cite[Theorem 5.10]{BindaMerici} we have\[
		\uExt^i_{\Sh_{\dNis}^{\ltr}}(\Z_{\ltr}(\P^1,\triv),W_m\Lambda^n)(T)\hookrightarrow H^i(\P^1_{\eta_T},W_m\Omega^n) = 0\quad \textrm{for }i\geq 2.
		\]
		It remains to prove the case $i=1$. By the cartesian square \eqref{eq:MV-P1} and Theorem \ref{thm:de-rham-witt-rsc-log} we have an isomorphism\[
		\uExt^1_{\Sh_{\dNis}^{\ltr}}(\Z_{\ltr}(\P^1,\triv),W_m\Lambda^n)\cong \uHom(\Z_{\ltr}((\P^1,0+\infty),i_1),\Log(W_m\Omega^n))
		\]
    For $X=(\ul{X},D)\in \SmlSm$, by the formula \cite[(6.0.2)]{shujilog} we have \[
    \uHom(\Z_{\ltr}((\P^1,0+\infty),i_1),\Log(W_m\Omega^n))(X) = \Hom_{\CI_{\Nis}}(h_{0,\Nis}^{\bcube,sp}(X,D_{\red}),\uHom(h_{0,\Nis}^{\bcube,sp}(\bcube^{(1)})/\Z,\widetilde{W_m\Omega^n}),
    \]
    see \cite[\S1.2 and Notation 2.1]{BindaRuellingSaito} for the notation. By \cite[Lemma 2.4 and Theorem 11.8]{BindaRuellingSaito} the right hand side is isomorphic to $\Log(W_m\Omega^{n-1})(X)$, which finishes the proof. 
	\end{proof}
	
	Let $G\in\logCI^{\ltr}(k,A)$, $X\in \SmlSm$ and $\cE\to X$ a vector bundle of rank $n+1$. Then if $k$ satisfies RS1 and RS2, the projective bundle formula of \cite[Theorem 8.3.5]{BPO} gives an equivalence\[
	R\Gamma_{\dNis}(\P(\cE),G)\simeq \bigoplus_{i=0}^n \Map_{\logDM}(M(X),\ul{\Map} ((M(\P^1),i_{0})^{\times i},G)),
	\]
	where the notation $(-,i_{0})$ is again the complement of the split analogous to \eqref{eq:deg-zero}. In particular, we deduce the projective bundle formula for $W_m\Lambda^n$, generalizing \cite{Gros1985}:
	\begin{equation}\label{eq:PBF-witt}
		R\Gamma(\P(\cE),W_m\Lambda^q) \simeq \bigoplus_{i=0}^{n} R\Gamma(X,W_m\Lambda^{q-i}),\qquad R\Gamma(\P(\cE),W\Lambda) \simeq \bigoplus_{i=0}^{n} R\Gamma(X,W\Lambda^{q-i}).
	\end{equation}
	Moreover, let $Z\subseteq \ul{X}$ be a smooth closed subscheme of codimension $d$ that intersects $|\partial X|$ transversally. Let $\rho\colon\widetilde{\ul{X}}\to \ul{X}$ be the blow up of $\ul{X}$ in $Z$ and consider the log scheme $\widetilde{X}=(\widetilde{\ul{X}},\partial X + E_Z)$ Then by \cite{BPO} there is a fiber sequence\[
	\Map(MTh(N_{Z}),G)\to R\Gamma(X,G)\to R\Gamma(\widetilde{X},G)
	\]
	On the other hand, there is a fiber sequence\[
	\P(N_Z)\to \P(N_Z\oplus \cO)\to MTh(N_Z)
	\]
	which by the projective bundle formula above implies that
	\[
	\Map_{\logDM}(MTh(N_{Z}),G)\simeq \Map_{\logDM}(M(Z),\ul{\Map}((M(\P^1),i_{0})^{\times d},G)).
	\]
	so we get a fiber sequence
	\begin{equation}\label{eq:purity}
		R\Gamma(Z,W_m\Lambda^{n-d})[-d]\to R\Gamma(X,W_m\Lambda^n)\to R\Gamma(\widetilde{X},W_m\Lambda^n)
	\end{equation}
	If $\partial X$ is trivial and $Z$ is a divisor, then the Gysin sequence above agrees with \eqref{eq:weight-es} of \cite{mokrane}. In general, it agrees with the Gysin sequence of \cite[Corollary 11.10 (2)]{BindaRuellingSaito} for a reduced modulus pair. 
	\begin{remark}
		The projective bundle formula and the Gysin sequences above were obtained in \cite{BindaRuellingSaito} without the assumption of resolutions of singularities, and are in fact independent of the results of \cite{BPO}.
	\end{remark}
	Finally, recall the Raynaud ring $R(k)$ of Ekedahl and the map n $\cD(R(k))$ (see \ref{ssec:logdRW}):
	\begin{equation}
		R\Gamma(X,W\Lambda)\widehat{\star^{L}} R\Gamma(Y,W\Lambda) \to R\Gamma(X\times Y,W\Lambda). \tag{\ref{eq:kunneth}}
	\end{equation}
	\begin{prop}\label{prop:kunneth}
		The map \eqref{eq:kunneth} is an equivalence in $\cD(R(k))$.
		\begin{proof}
			We proceed by induction on the dimension of $X\times Y$ and the number of components of $\partial (X\times Y)$.  If $\partial (X\times Y)$ is trivial, then $\partial X$ and $\partial Y$ are trivial so  \eqref{eq:kunneth} is an equivalence by \cite[Theorem 5.1.2]{Illusiefinite}. If $\dim(X\times Y)=0$, then $\partial(X\times Y)$ is again trivial so we can invoke \cite{Illusiefinite} again. In general, let $|\partial X|=D_1+\ldots D_r$ with $r\geq 1$ and consider $|\partial X|^{-}=D_2+\ldots D_r$. Let $X^-:=(X,\partial X^-)$, so that by \cite[I. Theorem 6.2]{EkedahlII} there is a commutative diagram in $\cD(R(k))$\footnote{Notice that, even though \cite[I. Theorem 6.2]{EkedahlII} is stated for the homotopy categories, the proof shows that $\widehat{\star^{L}}$ is indeed homotopy coherent}\[ \begin{tikzcd}
				R\Gamma(X^{-},W\Lambda)\widehat{\star^{L}} R\Gamma(Y,W\Lambda) \ar[d,"(1)"] \ar[r]&R\Gamma(X,W\Lambda)\widehat{\star^{L}} R\Gamma(Y,W\Lambda)\ar[d,"(2)"]\\
				R\Gamma(X^{-}\times Y,W\Lambda)\ar[r]&R\Gamma(X\times Y,W\Lambda)
			\end{tikzcd}
			\]
			By induction hypothesis, (1) is an equivalence, so to show that (2) is an equivalence it is enough to show that the map on the fibers is an equivalence. By \cite[Theorem 7.5.4]{BPO} together with \eqref{eq:purity} (or by using directly \cite[Corollary 11.10]{BindaRuellingSaito}), the fiber is\[
			R\Gamma(D_1,W\Lambda)\widehat{\star^{L}} R\Gamma(Y,W\Lambda)\to R\Gamma(D_1\times Y,W\Lambda)
			\]
			which is an equivalence by induction on dimension.
		\end{proof}
	\end{prop}
	\begin{remark}
		Let $X\in \SmlSm(k)$. By \cite[Corollary 2.4]{shujilog} we have that for all $0\leq d\leq \dim(\ul{X})$, $y\in \ul{X}^{(d)}$:\[
		H^j_{y}(X,W_m\Lambda^q) = 0\quad \textrm{if }j\neq d.
		\]
		This implies that the complex\[
		\ldots \to \bigoplus_{x\in \ul{X}^{(d)}}H^d_{x}(X,W_m\Lambda^n)\to\bigoplus_{x\in \ul{X}^{(d+1)}} H^{d+1}_{x}(X,W_m\Lambda^n)\to \ldots
		\]
		computes $H^d_{\dNis}(X,W_m\Lambda^q)$ since $W_m\Lambda^q$ is strictly div-invariant by Theorem \ref{thm:witt-local} and \cite{BPO}.
		This generalizes \cite[5.1]{Gros1985} to $W_m\Lambda^n$.
		
		Moreover, by \cite[Corollary 2.4]{shujilog}, for a fixed isomorphism $\epsilon_y\colon X_y^{h}\xrightarrow{\sim} \Spec(k(y)\{t_1\ldots t_d\})$ which sends $|\partial X_y^{h}|$ to the divisor $t_1^{e_1}\cdot\ldots \cdot t_d^{e_d}$ for some $e_i\in \{0,1\}$,\[
		H^q_{y}(X,W_m\Lambda^n) = 	\tau^{(e_1\cdots e_d)}W_m\Omega^n
		\]
		where the right hand side is defined recursively as:
		\begin{align*}
			&\tau^{(0)}W_m\Omega^n := \uHom((\Z_{\ltr}(\A^1,0),i_{1}),W_m\Lambda^n),\\
			&\tau^{(1)}W_m\Omega^n := \tau^{(0)}W_m\Omega^n/\uHom((\Z_{\ltr}(\P^1,0+\infty),i_{1}),W_m\Lambda^{n}),\\
			&\tau^{(e_1\cdots e_s)}W_m\Omega^n :=\tau^{(e_s)}\tau^{(e_1\cdots e_{s-1})}W_m\Omega^n,
		\end{align*}
	\end{remark}
	where $(-,i_{1})$ is as in \eqref{eq:deg-zero} the complement of the section $i_1\colon 1\to \A^1\to \P^1$.
	\section{The motivic integral \texorpdfstring{$p$}{p}-adic cohomology}
	
	Theorem \ref{thm:de-rham-witt-rsc-log} implies that the truncated de Rham--Witt complex $W_m\Lambda^{\bullet}$ is a complex of objects in $\logCI^{\ltr}(k,A)$, in particular its cohomology sheaves are strictly $\bcube$-invariant for the $\dNis$ topology and since they are coherent sheaves, also for the $\mathrm{d\et}$ and by \cite[Proposition 3.27]{Niziol} for the $\loget$ topology. This similarly to Remark \ref{rmk:log-crys} implies that the cohomology of the log-de Rham--Witt complex:\[
	\SmlSm(k) \to \cD(W(k))^{op}\qquad X\mapsto \holim_m R\Gamma(X,W_m\Lambda^{\bullet})
	\]
	factors through $\logDM(k,W(k))$ and $\logDMlet(k,W(k))$, via the functor:
	\begin{align*}
		R\Gamma_{\mathrm{crys}}\colon &\logDMlet(k,W(k))\to \cD(W(k))^{op}\\
		&M\mapsto \lim_m\Map(M,W_m\Lambda^{\bullet}).
	\end{align*}
	Recall that for $\tau\in \{\Nis,\et\}$, there is a composition of functors
	\[
	\cD(\Sh_{d\tau}^{\ltr}(k,A))\xrightarrow{\omega_\sharp}\cD(\Sh_{\tau}^{\tr}(k,A))\xrightarrow{L_{(\A^1,\tau)}} \mathcal{DM}^{\rm eff}_{\tau}(k,A)
	\]
	such that for all $X\in \SmlSm(k)$,\[
	L_{(\A^1,\tau)}L\omega_\sharp A_{\ltr}(X\times \bcube)\simeq L_{(\A^1,\tau)}A_{\tr}(X^{\circ}\times \A^1)\simeq L_{(\A^1,\tau)}A_{\tr}(X^{\circ}) \simeq L_{(\A^1,\tau)}L\omega_\sharp A_{\tr}(X^{\circ}),
	\]
	hence it induces $L^{\ltr}\omega_{\sharp}\colon \mathbf{log}\mathcal{DM}^{\rm eff}_{\tau}(k,A)\to \mathcal{DM}^{\rm eff}_{\tau}(k,A)$ such that $L^{\ltr}\omega_{\sharp}L_{(\bcube,d\tau)}\simeq L_{(\A^1,\tau)}L\omega_\sharp$. Again, in the case where $\tau=\rm \Nis$, by \cite[Proposition 2.5.7]{DoosungA1log} $L^{\ltr}\omega_\sharp$ is left adjoint to $\omega^*$, which is fully faithful and its essential image is the category of $\A^1$-local objects in $\logDM$.
	By composing $\omega^*$ with $R\Gamma_{\rm crys}$ above we get the following cohomology:
	\begin{equation}\label{eq:diagram-crys}
		\begin{tikzcd}
			R\Gamma_p\colon	\mathcal{DM}^{\rm eff}(k,\Z)\ar[r,"\omega^*"]&\logDM(k,\Z)\ar[r,"L_{\mathrm{l\et}}"]&\mathbf{log}\mathcal{DM}^{\eff}_{\loget}(k,\Z)\ar[r,"R\Gamma_{\mathrm{crys}}"]&\cD(W(k))^{op}
		\end{tikzcd}
	\end{equation}
	We are now ready to prove Theorem \ref{thm:man-intro-absolute}: the cohomology $R\Gamma_p$ has indeed by design a map $R\Gamma_p(X-D)\to R\Gamma_{\rm crys}(X,D)$, functorial in $(X,D)$, which is universal among $(\A^1,\Nis)$-local complexes $C$ equipped with a map $C(X-D) \to R\Gamma_{\rm crys}(X,D)$, functorial in $(X,D)$. 

    It remains to prove the comparison with rigid cohomology, which we now spell out. 

    If $X$ admits a smooth compactification $(\ol{X},D)$, then by the comparison of \cite[Cor.11.7 1)]{nakkajima} we have a map \[
    R\Gamma_p(X)[1/p]\to R\Gamma_{\rm crys}(\ol{X},D)[1/p]\simeq R\Gamma_{\rm rig}(X/K).
    \]
    This map is functorial in the following sense: if $Y\to X$ is a map and $(\ol{Y},D_Y)\to (\ol{X},D_X)$ is a map between smooth compactifications, then \[
    \begin{tikzcd}
        R\Gamma_p(X)[1/p]\ar[r]\ar[d] &R\Gamma_{\rm crys}(\ol{X},D_X)[1/p]\simeq R\Gamma_{\rm rig}(X/K)\ar[d]\\
        R\Gamma_p(Y)[1/p]\ar[r] &R\Gamma_{\rm crys}(\ol{Y},D_Y)[1/p]\simeq R\Gamma_{\rm rig}(Y/K)
    \end{tikzcd}
    \]
    is commutative.
    By \cite[\S 9]{nakkajima}, for all $X\in \Sm_k$ there is a simplicial object $(\ol{X}_\bullet,D_{\bullet})$ in $\SmlSm(k)$ such that $X_\bullet:= \ol{X_\bullet}-|D_\bullet|$ is a proper hypercovering of $X$. This is functorial in $X$: by \cite[Proposition 9.4]{nakkajima} there is a commutative diagram in $\Sm_k$
	\begin{equation}\label{eq:alteration}
		\begin{tikzcd}
			X'_\bullet\ar[d]\ar[r,"\ol{f}^{\circ}"]&X_\bullet\ar[d]\\
			X'\ar[r,"f"]&X,
		\end{tikzcd} 
	\end{equation}
	where the vertical arrows are proper hypercovers coming from simplicial objects $(\ol{X}_\bullet,D_{\bullet})$ and $(\ol{X'}_\bullet,D'_{\bullet})$ in $\SmlSm(k)$ as before and the map $\ol{f}^{\circ}$ comes from a map of hypercovers $\ol{f} \colon (\ol{X'}_\bullet,D_\bullet')\to (\ol{X}_\bullet,D_\bullet)$ in $\SmlSm(k)$. Since both $R\Gamma_{\rm rig}$ and $R\Gamma_p[1/p]$ are $(\A^1,\et)$-local, they satisfy $h$-hyperdescent by the usual argument of \cite[Theorem 4.1.12]{VoevTriangCat} using \cite[Theorem 4.0.3]{KellyThesis} as an input instead of \cite[Theorem 4.1.2]{VoevTriangCat} (see also \cite{Tsuzuki} for an independent proof that $R\Gamma_{\rm rig}$ satisfies $h$-hyperdescent), so we conclude by \eqref{eq:alteration} that we have a map\[
		\begin{tikzcd}
			R\Gamma_p(X)[1/p]\ar[r]\ar[d,equal]&R\Gamma_{\rm rig}(X)\ar[d,equal]\\
            R\Gamma_p(X_\bullet)[1/p]\ar[r]&R\Gamma_{\rm rig}(X_{\bullet})
		\end{tikzcd}
		\]
		functorial in $X$: this concludes the proof.
	\subsection{Assuming resolution of singularities}
	If we assume that $k$ admits resolution of singularities, then \eqref{eq:diagram-crys} can be made more explicit. 
	\begin{nota}\label{nota:RS}
		We say that $k$ satisfies resolutions of singularities if the following two properties are satisfied (see \cite[Definition 7.6.3]{BPO}):
		\begin{enumerate}
			\item[(RS1)] For any integral scheme $X$ of finite type over $k$, there is a proper birational morphism $Y \to X$ of schemes over $k$, which is an isomorphism on the smooth locus, such that $Y$ is smooth over $k$.
			\item[(RS2)] Let $f \colon Y \to X$ be a proper birational morphism of integral schemes over $k$ such that $X$ is smooth over $k$ and let $Z_1, \ldots , Z_r$ be smooth divisors forming a strict normal crossing divisor on $X$. Assume that
			\[f^{-1}(X-Z_1 \cup\ldots \cup Z_r)\to X-Z_1 \cup\ldots \cup Z_r\] is an isomorphism. Then there is a sequence of blow-ups\[
			X_n\xrightarrow{f_{n-1}} X_{n-1} \xrightarrow{f_{n-2}}\ldots \xrightarrow{f_{0}}X_0\simeq X
			\]
			along smooth centers $W_i \subseteq X_i$ such that
			\begin{enumerate}[label=\alph*.]
				\item the composition $X_n \to X$ factors through $f$,
				\item $W_i$ is contained in the preimage of $Z_1 \cup \ldots \cup Z_r$ in $X_i$,
				\item $W_i$ has strict normal crossing with the sum of the reduced strict transforms of \[
				Z_1,\ldots,Z_r,f_0^{-1}(W_0),\ldots,f_{i-1}^{-1}(W_{i-1})\]
				in $X_i$.
			\end{enumerate}
		\end{enumerate}
	\end{nota}
	
	Recall that by \cite[Proposition 8.2.8]{BPO}, for $X\in \Sm(k)$ with smooth Cartier compactification $\ol{X}$, for $M^V(X)\in \mathcal{DM}^{\mathrm{eff}}(k,\Z)$ the Voevodsky motive of $X$, we have\[ R\omega^*_{\Nis}M^V(X) = M(\ol{X},\partial X)
	\]
	where $\partial X$ is the log structure supported on $\ol{X}-X$.
	\begin{prop}\label{prop:comparison-ESS}
		If $k$ admits RS1 and RS2, for $X\in \Sm(k)$, we have an equivalence:\[
		R\Gamma_{p}(X) \simeq R\Gamma_{\mathrm{crys}}((\overline{X},X)/W(k))
		\]
        In particular, for $K=W(k)[1/p]$ we have an equivalence $R\Gamma_{p}(X)[1/p]\simeq R\Gamma_{\rm rig}(X/K)$.
		\begin{proof}
            The second part follows from the first and \cite[Cor.11.7 1)]{nakkajima}.
			Since $k$ admits resolutions of singularities, there exists $X\hookrightarrow \overline{X}$ a Cartier compactification and let $\partial X$ be the log structure supported on the divisor $\overline{X}-X$. Then the result follows from the equivalence $\omega^*M^V(X) = M(\overline{X},\partial X)$ of \cite[Proposition 8.2.8]{BPO}.
		\end{proof}
	\end{prop}
	\begin{remark} 
		As observed in \cite[Proposition 2.24]{ertl2021integral} there exist $X\in \Sm(k)$ with smooth Cartier compactification $(\ol{X},\partial X)$ and an étale hypercover $X_\bullet\to X$ with smooth Cartier compactification $(\ol{X}_\bullet,\partial X_{\bullet})$ such that\[
		H^i_{\rm crys}((X,\partial X)/W(k))\not \simeq H^i_{\rm crys} ((X_{\bullet},\partial X_{\bullet})/W(k)).
		\]
		This by Proposition \ref{prop:comparison-ESS} implies that $M_{\rm \loget}(\overline{X},\partial X) \not \simeq M_{\rm \loget}(\overline{X}_{\bullet},\partial X_{\bullet})$. In particular, this implies that the functor  \eqref{eq:diagram-crys}
		does {\bf not} factor through $\mathcal{DM}^{\rm eff}_{\et}(k,\Z)$, so the following diagram is {\bf not} commutative:\[
		\begin{tikzcd}
			\mathcal{DM}^{\rm eff}(k,\Z)\ar[r,"\omega^*"]\ar[dr,phantom,"{\not \circlearrowleft}"]\ar[d,"L_{\et}"]&\logDM(k,\Z)\ar[d,"L_{\loget}"]\\
			\mathcal{DM}^{\rm eff}_{\et}(k,\Z)\ar[r,"\omega^*_{\et}"]&\mathbf{log}\mathcal{DM}^{\rm eff}_{\loget}(k,\Z)
		\end{tikzcd}
		\]
		as $L_{\et}(M^V(X))\simeq L_{\et}M^V(X_\bullet)$. On the other hand, it commutes with rational coefficients.
		In fact, in this case, let $K=W(k)[1/p]$ and let $X\in \SmlSm(k)$ with $\ul{X}$ proper. By \cite[Cor.11.7 1)]{nakkajima} have that \[
		R_{\rm crys}(X)[1/p]\simeq R\Gamma_{\rm rig}(X^{\circ}/K),
		\] 
		and rigid cohomology factors through $\mathcal{DM}_{\rm \et}^{\rm eff}(k,\Q)$, there is in fact a commutative diagram\[
		\begin{tikzcd}
			\mathcal{DM}^{\mathrm{eff}}(k,\Q)\ar[r,"L_{\loget}\omega^*"]\ar[d,"\simeq"]&\mathbf{log}\mathcal{DM}^{\eff}_{\loget}(k,\Q)\ar[rrr,"{R\Gamma_{\mathrm{crys}}(-)[1/p]}"]&&&\cD(K). \\
			\mathcal{DM}^{\mathrm{eff}}_{\et}(k,\Q)\ar[urrrr,bend right = 20,"{R\Gamma_{\rm rig}(-/K)}"]
		\end{tikzcd}
		\]
	\end{remark}
	\subsection{}\label{ssec:tame}
	Finally, we show that $R\Gamma_p$ satisfies condition \ref{(iii)''} of the introduction. Notice that since $R\Gamma_p$ factors through $\logDMlet(k,\Z)$, $R\Gamma_p$ satisfies the following descent property:
	\begin{enumerate}
		\item[$\log$-(iii)] For all $(\ol{X},\partial X)\in \SmlSm(k)$ with $\ol{X}$ proper, for all $(\ol{U},\partial U)\to (\ol{X},\partial X)$ finite log-\'etale cover (in particular, strict), let $U\to X$ be the finite \'etale map on the open subschemes where the log structures is trivial. Then the \v Cech hypercover $U_\bullet \to X$ induces an equivalence\[
		R\Gamma_p(X)\to R\Gamma_p(U_{\bullet})
		\]
	\end{enumerate}
	The link between finite log \'etale and tame descent is summarized in the following result:
	\begin{lemma}\label{lem:tame-vs-log}
		If $X\in \Sm(k)$ has a smooth Cartier compactification $(\ol{X},\partial X)$, then for any $\phi\colon U\to X$ finite tame Galois cover of $k$-schemes, there exists a smooth Cariter compactifications $(\ol{U},\partial U)$ and a log \'etale cover $(\ol{U},\partial U)\to (\ol{X},\partial X)$ whose restriction $U\to X$ is $\phi$.
	\end{lemma}
	\begin{proof}
		Let $(\ol{X},\partial X)$ be a smooth Cartier compactification of $X$. By a combination of \cite[Theorem 1.1]{KerzSchmidt} and \cite[Theorem 7.6]{Illusielog}\footnote{This is stated without proof: for a sketch one can check \cite[3.3]{Stix}}, if $U\to X$ is a tame Galois cover there is a unique $(\ol{U}',\partial U')\to (\ol{X},\partial X)$ that is a Kummer \'etale Galois cover such that $U=\ol{U}'\times_{\ol{X}} X$. In particular, $(\ol{U}',\partial U')\in \lSm(k)$, so by \cite[Theorem 5.10]{NiziollogBU} (see also \cite[Proposition A.10.2]{BPO}) there is a log blow-up $(\ol{U},\partial U)\to (\ol{U}',\partial U')$ such that $(\ol{U},\partial U)\in \SmlSm(k)$. Since log blow-ups are dividing covers by \cite[A.11.1]{BPO}, $(\ol{U},\partial U)\to (\ol{X},\partial X)$ is a log \'etale cover.
	\end{proof}
	
	We remark that the counterexample to \ref{(iii)} given in \cite{ertl2021integral} is not a tame cover.
	
	We finish this section by deducing the following results for the cohomology $R\Gamma_{p}(-)$:
	\begin{thm}[Projective bundle formula]\label{thm:pbf}
		For $\cE\to X$ a vector bundle of rank $n+1$, we have\[
		R\Gamma_p(\P(\cE))\simeq \bigoplus_{i=0}^n R\Gamma_{p}(X)[-i]
		\]
		\begin{proof}
			From the projective bundle formula, we have\[
			M^V(\P(\cE))\simeq \bigoplus_{i=0}^n M^V(X)(i)[2i]
			\]
			hence by \cite[Proposition 8.2.8]{BPO}\[
			\omega^*(M^V(X)(i)[2i])\simeq M(X)\otimes ((M(\P^1,\triv)i_0))^{\otimes i},
			\]
			where $(M(\P^1,\triv),i_{0})$ is as in \eqref{eq:deg-zero} the orthogonal to $0\to \P^1$. By Corollary \ref{cor:dRW-P1} we have
			\[
			\ul{\Map}((M(\P^1,\triv),i_0),W_m\Lambda^\bullet) \simeq W_m\Lambda^{\bullet-1}[-1]
			\]
			which recursively gives:
			\[
			\ul{\Map}((M(\P^1,\triv),i_0)^{\otimes i},W_m\Lambda^\bullet) \simeq W_m\Lambda^{\bullet-i}[-i].
			\]
			The equivalence above is compatible with $W_{m+1}\Lambda^\bullet\to W_m\Lambda^\bullet$, so by taking the limit over $m$ we conclude the proof.
		\end{proof}
	\end{thm}
	
	\begin{thm}[Purity]\label{thm:purity}
		Let $X\in \Sm(k)$ and let $Z\subseteq X$ be a smooth closed subscheme of codimension $d$. Then we have a fiber sequence:\[
    		R\Gamma_p(Z)[-d]\to R\Gamma_p(X)\to R\Gamma_p(X-Z)
		\]
		\begin{proof}
			By \cite{MV} and Proposition \ref{prop:comparison-ESS} we have a fiber sequence\[
			R\Gamma_p(X)\to R\Gamma_p(X-Z)\to R\Gamma_p(\omega^*(M(Z)(d)[2d]).
			\]
			As before, by \cite[Proposition 8.2.8]{BPO} we have
			\[\Map(\omega^*(M^{V}(Z)(d)[2d]),W_m\Lambda^{\bullet})\simeq \Map(M(Z),W_m\Lambda^{\bullet-d}[-d])\simeq R\Gamma(Z,W_m\Lambda^\bullet)[-d].\]
			which by taking the homotopy limit over $m$ gives the result.
	\end{proof}\end{thm}
	\begin{thm}[K\"unneth formula]\label{thm:kunneth}
		For $X,Y\in \Sm(k)$, There is an equivalence in $\cD(R(k))$:
		\[
		R\Gamma_p(X)\widehat{\star^{L}} R\Gamma_p(Y)\simeq R\Gamma_p(X\times Y)
		\]
		\begin{proof}
			It follows from Proposition \ref{prop:kunneth} and \cite[Proposition 8.2.8]{BPO}
		\end{proof}   
	\end{thm}
	
	\bibliographystyle{alpha}
	\bibliography{bibMerici}
\end{document}